\documentclass[10pt,a4paper]{amsart}

\usepackage[latin1]{inputenc}
\usepackage{tikz}
\usepackage{caption}
\usepackage{amsmath}
\usepackage{amssymb}
\usepackage{graphics, float}
\usepackage{amsthm}
\usepackage{bm}
\usepackage{hyperref}
\hypersetup{colorlinks=true,citecolor=blue,linkcolor=blue}%
\usepackage{mathrsfs}
\usepackage{enumitem}
\usepackage{physics}
\DeclareMathAlphabet{\mathpzc}{OT1}{pzc}{m}{it}
\newtheorem{theorem}{Theorem}
\newtheorem{corollary}[theorem]{Corollary}
\newtheorem{lemma}[theorem]{Lemma}
\newtheorem{proposition}[theorem]{Proposition}

\theoremstyle{definition}
\newtheorem{definition}[theorem]{Definition}
\theoremstyle{remark}
\newtheorem{remark}[theorem]{Remark}

\newcommand{\R}{\mathbb{R}}
\newcommand{\T}{\mathbb{T}}
\newcommand{\N}{\mathbb{N}}
\newcommand{\C}{\mathbb{C}}
\newcommand{\D}{\mathbb{D}}
\newcommand{\Z}{\mathbb{Z}}
\newcommand{\F}{\mathcal{F}}
\newcommand{\vertiii}[1]{{\left\vert\kern-0.25ex\left\vert\kern-0.25ex\left\vert #1 
    \right\vert\kern-0.25ex\right\vert\kern-0.25ex\right\vert}}

\begin{document}

	\title[]{\large Analytic Structure of Stationary Flows of an Ideal Fluid with a Stagnation Point} 
	\author{Aleksander Danielski}
		\address{Department of Mathematics and Statistics, Concordia University\\
		Montreal, QC H3G 1M8, Canada\\}
		\email{al\_danielski@hotmail.com}

		\begin{abstract}
The flow of an ideal fluid possesses a remarkable property: despite limited regularity of the velocity field, its particle trajectories are analytic curves. In our previous work, this fact was used to introduce the structure of an analytic Banach manifold in the set of 2D stationary flows having no stagnation points. The main feature of our description was to regard the stationary flow as a collection of its analytic flow lines, parameterized non-analytically by values of the stream function $\psi$.

In this work, we adapt this description to the case of 2D stationary flows which have a single elliptic stagnation point. Namely, we consider flows in a domain bounded by the graph of analytic function $\rho = b(\varphi)$, where $(\rho,\varphi)$ are polar coordinates centred at the origin. The position $p$ of the stagnation point is an unknown and must be included in the solution. In polar coordinates $(r,\theta)$ centred at $p$, the flow lines are described by graphs of $r=a(\psi,\theta)$, where $a$ is a `partially-analytic' function (analytic in $\theta$, of finite regularity in $\psi$). The equation of stationary flow $\Delta \psi = F(\psi)$ is transformed to the quasilinear elliptic equation $\Xi(a) = F(\psi)$ for the flow lines. The analysis is complicated by the fact that the ellipticity of $\Xi$ degenerates at the stagnation point.

We introduce function spaces for the partially-analytic family of flow lines, modelled on the weighted Kondratev spaces, appropriate for the degenerate setting. The equation of stationary flow is thus regarded as an analytic operator equation in complex Banach spaces, with local solution given by the implicit function theorem. In particular, we show that near the circular flow of constant vorticity, the equation has unique solution $p, a(\psi,\theta)$ depending analytically on parameters $b(\varphi)$ and $F(\psi)$.
\end{abstract}
	
	\maketitle

{\large

\section{Analytic structure of 2D stationary flows}
\pagenumbering{arabic}

An ideal incompressible fluid consists of a continuum of particles, at all times filling the domain, subject only to the pressure the particles exert on each other to enforce the fluid's incompressibility.  The flow of such a fluid, that is, the evolution of its velocity field $\mathbf{u}(x,t)$ (which should be at all times tangent to the boundary) is governed by the incompressible Euler equations
\begin{equation}\label{euler}
\pdv{\mathbf{u}}{t} + \mathbf{u} \cdot \nabla \mathbf{u} + \nabla p = 0,  \qquad \nabla \cdot \mathbf{u} = 0,
\end{equation}
for which the scalar pressure field $p(x,t)$, due to incompressibility, is uniquely (up to additive constant) defined by $\mathbf{u}$.

By introducing the vorticity $\mathbf{\omega} = \nabla \times \mathbf{u}$ and taking the curl of \ref{euler}, one obtains the Euler equations in vorticity form. For the two-dimensional fluid, the vorticity always points normal to the flow and is thus taken as a scalar. In this case, the Euler equations in vorticity form are
\begin{equation}
\pdv{\omega}{t} + \mathbf{u} \cdot \nabla \omega = 0, \qquad \omega = \nabla^{\perp} \cdot \mathbf{u} \qquad \nabla \cdot \mathbf{u}=0 .
\end{equation}
In particular, for the 2D fluid, the vorticity is transported along the particle trajectories by the flow.

Long-time existence and uniqueness of classical solutions of the initial value problem for the 2D fluid was established nearly a century ago, starting with the work of Wolibner (\cite{Wo}). These classical solutions possess the remarkable property that despite their limited regularity, say $\mathbf{u} \in H^m$ with $m > n/2 + 1$ (where $n$ is the dimension of the fluid domain), the particle trajectories $x(t)$, which satisfy $\dot{x}(t) = \mathbf{u}(x(t),t)$, are real analytic curves. This recent discovery was first proven by Serfati (\cite{Se}), and later independently by numerous authors (see \cite{Sh1}, \cite{Co}, \cite{FrZh}, \cite{InKaTo}, \cite{In}, \cite{Na}). 

The simplest class of solutions to the 2D Euler equations are the stationary (time-independent) flows. These are divergence-free vector fields $\mathbf{u}$, tangent to the boundary, along whose flow lines the vorticity is constant. The divergence-free condition ensures the existence of the stream function $\psi$ (unique up to additive constant), such that $\mathbf{u} = \nabla^{\perp} \psi$ and $\omega = \Delta \psi$. Its critical points correspond to stagnation points of the flow and its level lines correspond to the flow lines. The boundary condition of \ref{euler} implies $\psi$ should be constant along each component of the boundary. If $\psi$ is strictly monotone transversally to the flow lines, then its values uniquely parameterize the set of the flow lines. Stationary solutions of this form satisfy equation
\begin{equation}\label{stationary}
\Delta \psi = F(\psi), \qquad \psi \rvert_{\Gamma_i} = c_i,
\end{equation}
where $F(\psi)$ is the prescribed vorticity along the flow lines and $\Gamma_i$ are the components of the boundary of the domain. The solutions of the above equation possess an exotic feature: though the regularity of $\psi$ may be finite, its level lines are analytic curves (in the stationary flow, the particle trajectories and flow lines coincide). We call this property `partial-analyticity' of $\psi$.

The study of the local structure of the set of 2D stationary flows was introduced by Choffrut and \v{S}ver\'{a}k (\cite{ChSv}). They showed that under some non-degeneracy conditions on the reference solution, the set of
nearby stationary solutions form a smooth manifold. In their description, stationary flows in annular domains $\Omega$ having no stagnation points are locally parameterized by the distributions of vorticity $A_\omega(\lambda) =  \lvert \{ x \in \Omega \; : \; \omega(x) < \lambda \} \rvert$. Due to loss of derivatives in the linearization of \ref{stationary}, their approach necessitates working in the class of smooth functions and thus using the Nash-Moser-Hamilton implicit function theorem for Fr\'{e}chet spaces.

In our previous work (\cite{Da}), we introduced a different local description of the set of 2D stationary flows. In our approach, we incorporated the partially-analytic structure of the stationary solutions by regarding the flow as a collection of its analytic flow lines. We considered the flow in a periodic channel $ \{ (x,y) : x \in \T, \; f(x) < y < g(x) \}$ bounded by the graphs of analytic functions $y=f(x)$ and $y=g(x)$. If the flow has no stagnation points then the values of $\psi$ parameterize the family of flow lines, described by the graphs of function $y=a(x,\psi)$. By restricting values of $\psi$ to say $[0,1]$, we treat it as an independent variable and rewrite \ref{stationary} as an equation for the flow lines $a(x,\psi)$. In the new coordinates, which were already introduced by von Mises {\cite{Mi}) and Dubreil-Jacotin (\cite{Du}), we obtain for the velocity and vorticity the expressions
\begin{equation}
\mathbf{u}(\psi,\theta) =(\psi_y, -\psi_x) = \frac{(1,a_x)}{a_\psi},
\end{equation}
\begin{equation}
\Phi(a) = \Delta \psi  = \frac{-1}{a_\psi}a_{xx} + \frac{2a_x}{a_\psi^2}a_{x \psi} - \frac{1+a_x^2}{a_\psi^3}a_{\psi \psi},
\end{equation}
and the quasilinear elliptic boundary value problem
\begin{equation}
\Phi(a) = F(\psi), \qquad a(x,0) = f(x), \qquad a(x,1) = g(x),
\end{equation}
for the stationary flow lines $a(x,\psi)$, where $(x,\psi) \in \T \times [0,1]$. The ellipticity of $\Phi$ is non-degenerate away from stagnation points.

We quantify the analyticity of flow lines by extending $x \in \T$ to the complex periodic strip $\T_\sigma = \T \times i(-\sigma, \sigma)$. The Paley-Wiener theorem characterizes functions analytic in this strip by their complex singularities on the strip boundaries. We defined the Banach space $X_\sigma^m(\T)$ of complex analytic flow lines in $\T_\sigma$ with norm
\begin{equation*}
\lVert a(x) \rVert_{X_\sigma^m}^2 = \lVert a(\cdot + i \sigma) \rVert_{H^m(\T)}^2 + \lVert a(\cdot - i \sigma) \rVert_{H^m(\T)}^2,
\end{equation*}
and the space $Y_\sigma^m(\T \times [0,1])$ of partially-analytic families of complex flow lines with norm
\begin{equation*}
\lVert a(x,\psi) \rVert_{Y_\sigma^m}^2 = \lVert a(\cdot + i \sigma, \cdot) \rVert_{H^m(\T \times [0,1])}^2 + \lVert a(\cdot - i \sigma, \cdot) \rVert_{H^m(\T \times [0,1])}^2	.
\end{equation*}
For $m$ sufficiently high, an application of the analytic implicit function theorem in complex Banach spaces proves that the vorticities $F(\psi) \in H^{m-2}[0,1]$ and domains bounded by graphs of analytic functions $f(x), g(x) \in X_\sigma^{m-1/2}$ analytically parameterize the Banach manifold of stationary flows $a(x,\psi)$ in $Y_\sigma^m$ near the constant parallel flow $\psi = y$. The behaviour of the complex singularities of the analytic flow lines is controlled by the complex singularities of the boundary.

Both approaches discussed above consider only flows without stagnation points. In this work, we show that the general principle of treating a stationary flow as a family of analytic flow lines can be extended to produce a local description of stationary flows having a stagnation point. In particular, we consider flows on a simply connected domain close to the unit disk having a single non-degenerate elliptic stagnation point. Certain obstacles not present in the absence of stagnation points must be overcome. In polar coordinates $(r,\theta)$ centred at the stagnation point, the flow lines are expressed as graphs of partially-analytic functions $r=a(\psi,\theta)$. The position of the stagnation point is an unknown and must be incorporated into the solution. This fact yields a boundary condition which is nonlinear. The stationary flows are governed by a quasilinear elliptic boundary value problem whose ellipticity degenerates at the stagnation point. The main difficulty is in carefully constructing spaces of partially-analytic flow lines which correctly describe the non-degenerate stagnation point while satisfying the requirements of the implicit function theorem. We introduce such function classes, modelled on the Kondratev spaces, and show the relevant degenerate elliptic equations are well-posed in them. The result is an analytic parameterization of the stationary solutions near the circular flow $\psi = x^2 + y^2$ by the prescribed vorticity and domain.


\section{Flow lines coordinates about an elliptic stagnation point}

We start by introducing coordinates for the flow lines and derive the boundary value problem for the stationary flow in these coordinates. Next, we construct the appropriate function spaces for the family of flow lines around a non-degenerate fixed point, as well as spaces for the parameters of the problem, and discuss their properties.

The prototypical stationary flow having a single non-degenerate elliptic fixed point is described by the stream function $\psi = x^2 + y^2$ in say the unit disk $\D$, our logical reference solution. This flow has constant vorticity $F(\psi) = 4$ and describes the motion of a fluid rotating as a rigid body. The flow lines, which are concentric circles about the origin, are transversally parameterized by values of $\psi \in [0,1]$, from the fixed point to the boundary. Now consider a small deformation of the flow lines, such that they remain parameterized by the same values of $\psi$, namely $\psi=0$ at the fixed point and $\psi = 1$ at the boundary. The stagnation point may translate and the concentric circles around it deform. It is known that a stationary flow in a disk having a single fixed point must be circular and thus the fixed point must be positioned at the disk's centre. We expect this rigidity of the fixed point's position relative to the parameters to hold for general non-circular flows as well. We thus introduce  $p = (p_x, p_y) \in \R^2$, the position of the fixed point, as an unknown in the problem. Now letting $(r,\theta)$ be polar coordinates centred at $p$, the flow lines can be expressed as the graphs of a family of polar functions $r=a(\psi,\theta)$, so long as the deformations are not too large. The flow lines are thus parameterized in the computational domain $(\psi,\theta) \in [0,1] \times \T = \Pi$.

Inverting the Jacobian of transformation $r=a(\psi,\theta)$ and applying the chain rule, one can derive expressions for velocity and vorticity in the new coordinates. One finds
\begin{equation}
\mathbf{u}(\psi,\theta) = \frac{1}{a_\psi} \big( \frac{a_\theta}{a}, 1 \big) \qquad \text{ in $(\hat{r}, \hat{\theta})$ coordinates,}
\end{equation}
and
\begin{equation}\label{nonlinear}
\Delta \psi = \Xi(a) = -\frac{1}{a_\psi^3} \Big( 1 + \frac{a_\theta^2}{a^2} \Big) a_{\psi \psi} + 2 \Big( \frac{a_\theta}{a^2 a_\psi^2} \Big) a_{\psi \theta} - \Big( \frac{1}{a^2 a_\psi}\Big) a_{\theta \theta} + \frac{1}{a a_\psi}. 
\end{equation}
Observe that the flow lines degenerate to a meaningful stagnation point as $\psi \to 0^+$ if $a(\psi,\theta) \to 0$ and $\lvert a_\psi(\psi,\theta) \rvert \to \infty$. In particular, the coordinate change transforms critical points of $\psi \sim x^2+y^2$ to square root singularities $a(\psi,\theta) \sim \psi^{1/2}$. 

The vorticity $\Xi(a)$ is a second order quasilinear differential operator of the flow lines. Such operators of form $A a_{\psi \psi} + 2B a_{\psi \theta} + C a_{\theta \theta} + D$, are elliptic if $AC-B^2 > 0$. A straightforward calculation shows $AC-B^2 = (a^2 a_\psi^4)^{-1}$. Restricting to non-degenerate stagnation points ($a \sim \psi^{1/2}$), one sees that $AC-B^2 \to 0$ as $\psi \to 0^{+}$. We conclude that ellipticity of $\Xi$ degenerates along the boundary $\psi=0$ of the computational domain $\Pi$, where the fixed point is described.

The boundary conditions are complicated by two facts: first is the unknown position of the fixed point and second is that fixing the values of $\psi$ to the interval $[0,1]$ yields an overdetermined problem. We address these issues as follows. Suppose $a(\psi,\theta)$ is a stationary flow in a prescribed simply connected domain $\Omega$. Then the graph of $a(1,\theta)$ corresponds to $\partial \Omega$. However, $(r,\theta)$ coordinates depend on position of fixed point $p$ which is an unknown, and thus cannot be used for the prescribed parameters. Instead introduce polar coordinates $(\rho,\varphi)$ centred at the origin. If $\Omega$ is close enough to the disk, then $p$ is close enough to the origin, so $\partial \Omega$ can be represented as a graph of a polar function in both $(r,\theta)$ and $(\rho, \varphi)$ coordinates. We can thus prescribe $\Omega$ as the domain bounded by the graph of function $\rho = b(\varphi)$. The consequence however is that we must pass from $(\rho,\varphi)$ to $(r,\theta)$ coordinates in the boundary condition.

The second obstacle stems from the observation that specifying $\psi = 0$ at the fixed point yields an interior point condition on \ref{stationary}, making the problem overdetermined by one degree of freedom. As a result, only a co-dimension one subset of the parameters (vorticity and domain) yield a well-posed problem, the remaining are incompatible. To work around this, we introduce an extra degree of freedom $R$ in the solution, which radially rescales the graph of $r=a(\psi,\theta)$ relative to the fixed point $p$. Now given some prescribed vorticity and boundary $\rho = b(\varphi)$, we are looking for a fixed point $p$ and a family of flow lines $a(\psi,\theta)$ such that upon rescaling by $R$, the graph of $r=Ra(1,\theta)$ equals the graph of $\rho = b(\varphi)$. Refer to the following figure of the boundary condition.

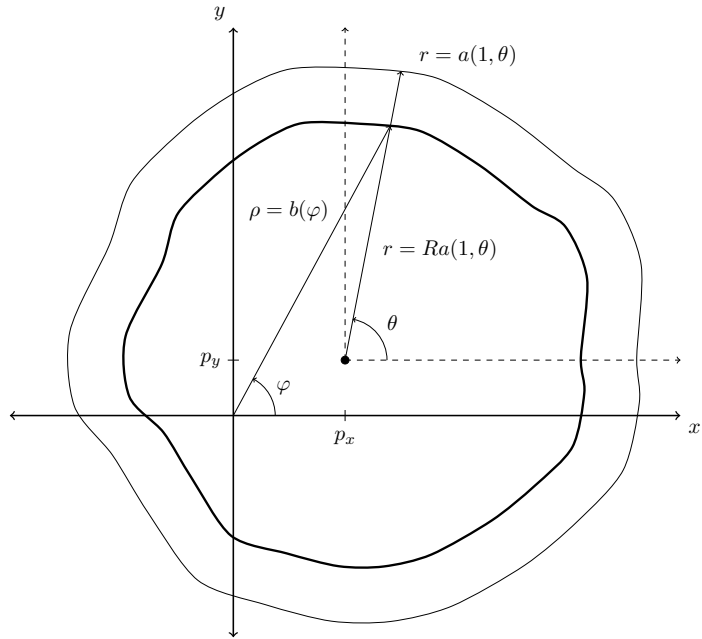
\begin{figure}[H]
\scalebox{0.735}
{
\centering
\begin{tikzpicture}
	\draw[thick,<->] (-4,0) -- (8,0) node[anchor=north west] {$x$};
	\draw[thick,<->] (0,-4) -- (0,7) node[anchor=south east] {$y$};
	\filldraw [black] (2,1) circle (2pt);
	\draw[dashed,->] (2,1) -- (8,1);
	\draw[dashed, ->] (2,1) -- (2,7);
	\draw[ultra thin] (2,-0.1) -- (2,0.1);
	\node at (2,-0.4) {$p_x$};
	\draw[ultra thin] (-0.1,1) -- (0.1,1);
	\node at (-0.4,1) {$p_y$};
	\draw[->] (2,1) -- (3, 6.21);
	\draw[ultra thin, ->] (2.75, 1) arc (0:79.16:0.75); 
	\node at (2.85,1.67) {$\theta$};
	\node at (3.7,3) {$r=Ra(1,\theta)$};
	\node at (4.2,6.5) {$r=a(1,\theta)$};
	\draw[->] (0,0) -- (2.81, 5.23);
	\draw[ultra thin, ->] (0.75, 0) arc (0:61.75:0.75);
	\node at (0.9, 0.5) {$\varphi$};
	\node at (1,3.7) {$\rho = b(\varphi)$};
	\pgfmathsetseed{33}
	\draw[very thick, domain=-0:350, smooth cycle, variable=\t] plot ({2+ 4*cos(\t)+rnd*0.5},{1+4*sin(\t)+rnd*0.5});
	\pgfmathsetseed{33}
	\draw[ultra thin, domain=-0:350, smooth cycle, variable=\t] plot ({2+ 5*cos(\t)+rnd*0.5},{1+5*sin(\t)+rnd*0.5});
\end{tikzpicture}
}
\caption*{The inner deformed circle represents the prescribed boundary flow line, defined by the graph of $\rho = b(\varphi)$. We seek a family of flow lines $a(\psi,\theta)$ about some fixed point $p$ which when rescaled by some $R$, matches the boundary at $\psi =1$. The unscaled flow line $r=a(1,\theta)$, depicted by the outer deformed circle, defines a new domain of flow of the same shape as the prescribed one, of a radius compatible with the prescribed vorticity.}
\end{figure}

We obtain the following equations relating $b(\varphi)$, $R$ and $a(1,\theta)$:
\begin{equation*}
b(\varphi)\cos \varphi = p_x +Ra(1,\theta)\cos \theta, \qquad b(\varphi)\sin\varphi = p_x +Ra(1,\theta)\sin \theta.
\end{equation*}
Let us introduce $\varphi = \arctan(y,x)$ as the function onto $\T$ whose values are given by the angle between plane vector $(x,y)$ and the $x$-axis (as opposed to the usual inverse tangent function defined only on the half plane). Then the above equations can be reduced to a single condition on the boundary, written $B(b,R,p,a) = 0$, where
\begin{multline}\label{boundary}
B(b,R,p,a) = -b^2\Big( \arctan\big(p_y + Ra(1,\theta)\sin\theta, p_x + Ra(1,\theta)\cos\theta\big) \Big)\\
 + R^2a^2(1,\theta) + 2Ra(1,\theta)\big(p_x \cos \theta + p_y \sin \theta) + p_x^2+p_y^2.
\end{multline}

Including the mentioned correction to the boundary condition, the transformed equation of stationary flow \ref{stationary} reads:
\begin{equation}\label{NLBVP}
\Xi(a) = F(\psi), \qquad B(b,R,p,a) = 0.
\end{equation}
 Given parameters $b(\varphi)$ and $F(\psi)$ defined in $\T$ and $[0,1]$ respectively, the equation is to be solved for $R \in \R$, $p \in \R^2$ and $a(\psi,\theta)$ in domain $\Pi = [0,1] \times \T$. 

We look for solutions near the reference flow $\psi = x^2 + y^2$, which in our coordinates is given by $R=1$, $p=0$, $a(\psi,\theta) =  \psi^{1/2}$. These solutions should be taken in some appropriate function space with the following properties. In a sufficiently small neighbourhood of the reference, the functions should describe families of flow lines degenerating to a unique fixed point at $\psi =0$. In particular, for the fixed points to be non-degenerate, the functions should behave asymptotically like $a \sim \psi^{1/2}$ as $\psi \to 0^+$. For such functions, ellipticity of operator $\Xi$ degenerates at $\psi=0$ so the function space should be suited to this scenario. Finally the functions should incorporate the partial-analytic structure of the flow lines, namely, $\theta \to a(\psi,\theta)$ should be analytic functions. For insight on the degeneracy, let us look at the linearization of \ref{NLBVP} with respect to the reference solution. After dropping some multiplicative factors, one obtains the following linear problem:

\begin{equation} \label{LBVP}
\begin{cases}
\psi^{-1/2} \Bigl[ \psi^2 \pdv[2]{\psi} + 2 \psi \pdv{\psi} + \frac{1}{4} ( I + \pdv[2]{\theta} ) \Bigr] u(\psi,\theta) = f(\psi,\theta) \\
R + ( \frac{p_x-ip_y}{2} )e^{i\theta}  + ( \frac{p_x+ip_y}{2} ) e^{-i\theta} + u(1,\theta)= g(\theta),
\end{cases}
\end{equation}
to be solved for $R$, $p_x$, $p_y$ and $u(\psi,\theta)$ given data $f(\psi,\theta)$ and $g(\theta)$.

Modulo the factor of $\psi^{-1/2}$, the degeneracies are of form $\psi \pdv{\psi}$. Such order-one degeneracies are typical of elliptic equations on manifolds with conical singularities. Typically in such a context, the singular point is deleted and the manifold is stretched along the removed singular point (consider a cone blown up to a cylinder). In doing so, the problem is transformed to a degenerate  elliptic equation on a manifold with boundary. In his seminal paper (\cite{Ko}), Kondratev introduced weighted Sobolev spaces in which such boundary value problems are Fredholm. Such a formulation is similar to our own, where a family of simple closed curves degenerate to a single fixed point. Adapted to our scenario, where the degeneracy occurs along $\psi=0$, we define the Kondratev spaces $K_\gamma^m[0,1]$ and $K_\gamma^m(\Pi)$ with norms
\begin{equation}
\left\lVert u(\psi) \right\rVert_{K_\gamma^m[0,1] }^2 = \sum_{p=0}^m \left\lVert \psi^{p - \gamma} D^p u(\psi)  \right\rVert_{L^2 [0,1] }^2
\end{equation}
and
\begin{equation}
\left\lVert u(\psi,\theta) \right\rVert_{K_\gamma^m(\Pi) }^2 = \sum_{p +q=0}^m \left\lVert \psi^{p - \gamma} \partial_\psi^p \partial_\theta^q u(\psi,\theta)  \right\rVert_{L^2 (\Pi) }^2,
\end{equation}
respectively. The parameter $\gamma$ controls the behaviour of the asymptotics as $\psi \to 0^+$ and the weight in the norms is homogeneous with respect to derivatives in $\psi$. A detailed overview of these spaces can be found in \cite{BoKo}. We list a few important properties.

\begin{proposition} \label{Kondratev properties} \
\begin{itemize}
\item $u(\psi,\theta) \to \psi^{\alpha}u(\psi,\theta) : K_\gamma^m(\Pi) \to K_{\gamma+\alpha}^m(\Pi)$ defines an isomorphism.
\item $u(\psi,\theta) \to \partial_\psi u(\psi,\theta) : K_\gamma^m(\Pi) \to K_{\gamma-1}^{m-1}(\Pi)$ is bounded.
\item $u(\psi,\theta) \to \partial_\theta u(\psi,\theta) : K_\gamma^m(\Pi) \to K_{\gamma}^{m-1}(\Pi)$ is bounded.
\item $u(\psi,\theta) \to u(1,\theta) : K_\gamma^m(\Pi) \to H^{m-1/2}(\T)$ is bounded.
\end{itemize}
\end{proposition}

\begin{proposition} \label{Morrey Kondratev} \ \\
There exists $C>0$ depending on $\gamma$, $m$, $k$ for which
\begin{itemize}
\item $\lVert \partial_\psi^ku(\psi,\cdot) \rVert_{H^{m-k-1/2}(\T)} \leq C\psi^{\gamma-k-1/2} \lVert u \rVert_{K_\gamma^m(\Pi)}$ for $m-k>1/2$.
\item $\lvert \partial_\psi^ku(\psi,\theta) \rvert \leq C\psi^{\gamma-k-1/2} \lVert u \rVert_{K_\gamma^m(\Pi)}$ for $m-k>1$.
\end{itemize}
\end{proposition}
The latter proposition clarifies the dependence of behaviour as $\psi \to 0^+$ on the parameter $\gamma$. In particular, functions in $K_\gamma^m(\Pi)$ are continuous for $m>1$ and $\gamma>1/2$ and in this case they necessarily vanish as $\psi \to 0^+$.

For our purposes, the above spaces are not quite adequate. One can easily check that for any choice of $\gamma$ such that $\psi^{1/2} \in K_\gamma^m$, every neighbourhood of this function necessarily contains perturbations of lower order, say $\psi^{\mu}$, where $\mu < 1/2$. Such perturbations present an obstruction to the description of a family of flow lines uniquely degenerating to a non-degenerate fixed point at $\psi=0$. Instead we should consider only angular and higher order perturbations of $\psi^{1/2}$ (or more generally, $\psi^{\lambda}$). We thus introduce the following spaces of functions, asymptotically behaving like $\psi^\lambda$ as $\psi \to 0^+$:

\begin{definition}[Spaces of $\lambda$-order asymptotics as $\psi \to 0^+$]
\begin{equation*}
J_{\lambda, \gamma}^m[0,1] = \left\{ a(\psi) = v \psi^\lambda + w(\psi) \colon v \in \R, w(\psi) \in K_{\lambda + \gamma}^m[0,1] \right\}
\end{equation*}
\begin{equation*}
J_{\lambda, \gamma}^m(\Pi) = \left\{ a(\psi,\theta) = v(\theta) \psi^\lambda + w(\psi,\theta) \colon v(\theta) \in H^m(\T), w(\psi,\theta) \in K_{\lambda + \gamma}^m(\Pi) \right\}
\end{equation*}
with norms induced by the direct sum $J_{\lambda, \gamma}^m[0,1] \cong \R \oplus K_{\gamma+ \lambda}^m[0,1]$ and $J_{\lambda, \gamma}^m(\Pi) \cong H^m(\T) \oplus K_{\gamma+ \lambda}^m(\Pi)$.
\end{definition}
These direct sums are well defined when $\lambda \geq 1/2$, which is equivalent to the condition that the remainder terms $w \in K_{\gamma+\lambda}$ exclusively consist of asymptotics of order greater than $\psi^{\lambda}$.

Following the previous work (\cite{Da}), we introduce the partially-analytic (analytic in $\theta$) structure on the subset of these functions with analytic continuations from $\T$ to the strip $\T_\sigma = \T \times i(-\sigma, \sigma)$. For completeness, we remind first the space of individual complex analytic flow lines $X_\sigma^m(\T)$:

\begin{definition}\ \\
Let $a(x)$ be a function on the circle $\T$ with Fourier series $\sum_k \hat{a}_k e^{ikx}$. 
	\begin{itemize}
		\item $X_\sigma^m(\T)$ is the space of functions $a(x)$ on the circle with norm
			\begin{equation*}
			\lVert a(x) \rVert_{X_\sigma^m}^2 =\lVert \F_{k \to x}^{-1} \{ \hat{a}_k e^{\sigma \lvert k \rvert} \} \rVert_{H^m(\T)}^2 =  \sum_k (1+k^2)^m e^{2\sigma \lvert k \rvert} \lvert \hat{a}_k \rvert^2.
			\end{equation*}
		\item Equivalently, $X_\sigma^m(\T)$ is the space of analytic functions 
		\begin{equation*}
		z= x+it \to a(z): \T_\sigma \to \C
		\end{equation*}
		with norm
			\begin{equation*}
			\lVert a(z) \rVert_{X_\sigma^m}^2 = \lVert a(\cdot + i \sigma) \rVert_{H^m(\T)}^2 + \lVert a(\cdot - i \sigma) \rVert_{H^m(\T)}^2.	
			\end{equation*}
	\end{itemize}
\end{definition}

Next, a straightforward adaptation of the Paley-Wiener theorem to Banach-valued analytic functions gives us two equivalent characterizations for the partially-analytic family of flow lines.
\begin{definition}\
\begin{itemize}
\item 
$J_{\lambda, \gamma}^{m, \sigma}(\Pi) = \left\{ a(\psi,\theta) \in J_{\lambda, \gamma}^m(\Pi) : \F_{k \to \theta}^{-1} \{ \hat{a}_k(\psi)e^{\sigma \lvert k \rvert} \} \in J_{\lambda, \gamma}^m(\Pi) \right\}$ \\
with norm
\begin{equation*}
\lVert a(\psi,\theta) \rVert_{J_{\lambda, \gamma}^{m,\sigma}(\Pi)} = \lVert \F_{k \to \theta}^{-1} \{ \hat{a}_k(\psi)e^{\sigma \lvert k \rvert} \} \rVert_{J_{\lambda,\gamma}^m(\Pi)}.
\end{equation*}
\item 
$J_{\lambda,\gamma}^{m,\sigma}(\Pi)$ is the space of holomorphic functions
\begin{equation*}
z = x + it \to a(\cdot, z) : \T_\sigma \to J_{\lambda,\gamma}^m[0,1]
\end{equation*}
with norm
\begin{equation*}
\lVert a(\psi,z) \rVert_{J_{\lambda,\gamma}^{m,\sigma}(\Pi)}^2 = \lVert a(\cdot, \cdot + i\sigma) \rVert_{J_{\lambda, \gamma}^m(\Pi)}^2 + \lVert a(\cdot, \cdot - i\sigma) \rVert_{J_{\lambda, \gamma}^m(\Pi)}^2.
\end{equation*}
\end{itemize}
\end{definition}
The parameters defining the space $J_{\lambda, \gamma}^{m,\sigma}(\Pi)$ can be summarized as follows:
\begin{itemize}
\item $\lambda$ describes the leading order asymptotics as $\psi \to 0^+$.
\item $\gamma \geq 1/2$ defines the scale of remainder term asymptotics.
\item $m$ is the usual isotropic regularity scale.
\item $\sigma$ quantifies the partial-analyticity in $\theta$.
\end{itemize}

Now we take $J_{1/2,\gamma}^{m,\sigma}(\Pi)$ to be the space of functions $r=a(\psi,\theta)$ which in a neighbourhood of $\psi^{1/2}$ describe the partially-analytic flow lines about a non-degenerate fixed point. The restriction map $a(\psi,\theta) \to a(1,\theta)$ is bounded in $J_{1/2,\gamma}^{m,\sigma}(\Pi) \to X_\sigma^{m-1/2}(\T)$. The natural target space for \ref{NLBVP} and \ref{LBVP} is $J_{0,\gamma}^{m-2,\sigma}(\Pi)$. Functions in this space are continuous for $m>3$ and $\gamma > 1/2$. In the next chapter, we will find that this target space is inadequate due to the presence of a two-dimensional cokernel in the linearized problem which must be factored out to establish the required isomorphism. We thus define
\begin{equation}
\widetilde{J}_{0,\gamma}^{m,\sigma}(\Pi) = \left\{ u = v(\theta) + w(\psi,\theta) \in J_{0,\gamma}^{m,\sigma}(\Pi) : \int_\T v(\theta)e^{\pm 2i\theta} \dd{\theta} = 0 \right\}
\end{equation}

The complex vorticities $F(\psi)$ are taken in the space $J_{0,\gamma}^{m-2}[0,1]$, which consists of continuous perturbations of the constant function by higher order perturbations. It naturally embeds into $\widetilde{J}_{1/2,\gamma}^{m-2,\sigma}(\Pi)$  by constant continuation along $\theta$, that is, $F(\psi) = F(\psi,\theta)$. 

We will see that to accommodate the coordinate change in the boundary condition of \ref{NLBVP}, we must prescribe the analytic boundary data $\rho = b(\varphi)$ on a complex strip $\T_\tau = \T \times i(-\tau,\tau)$ that is slightly wider than used to describe the solutions. In particular, it suffices to take $b(\varphi) \in \mathrm{H}(\T_\tau)$ for any $\tau > \sigma$, where $\mathrm{H}(\T_\tau)$ is any Banach space of holomorphic functions in $\T_\tau$.

\begin{definition}[Complexified nonlinear boundary value problem] \label{complex operator equation}\ \\
Define operator
\begin{equation*}
(F,b,R,p,a) \to \left( \Xi(a) - F(\psi), B(b,R,p,a) \right)
\end{equation*}
in spaces
\begin{equation*}
J_{0,\gamma}^{m-2}[0,1] \times \mathrm{H}(\T_\tau) \times \C^3 \times J_{1/2,\gamma}^{m,\sigma}(\Pi) \to \widetilde{J}_{0,\gamma}^{m-2,\sigma}(\Pi) \times X_\sigma^{m-1/2}(\T).
\end{equation*}
Solutions to \ref{NLBVP} are the zeros of the above operator. 

In other words, we consider the problem to find $R \in \C$, $p \in \C^2$, $a(\psi,\theta) \in J_{1/2,\gamma}^{m,\sigma}(\Pi)$ given parameters $F(\psi) \in J_{0, \gamma}^{m-2}[0,1]$ and $b(\varphi) \in \mathrm{H}(\T_\tau)$. At least we look for solutions near the reference given by  $R=1$, $p=0$, $a(\psi,\theta) = \psi^{1/2}$, $b(\varphi) = 1$ and $F(\psi) = 4$.
\end{definition}

The main tool to solve this problem will be the analytic implicit function theorem in complex Banach spaces, which gives condition under which an operator equation with parameter has a unique local solution.

\begin{theorem}[Analytic Banach implicit function theorem] \ \\
Let $X,Y,Z$ be complex Banach spaces and $f: X \times Y \to Z$ be an analytic map in a neighbourhood of $(x_0,y_0) \in X \times Y$. Suppose $f(x_0,y_0) = 0$ and $\pdv{f}{y} (x_0,y_0) : Y \to Z$ is an isomorphism. Then there exists a neighbourhood of $(x_0,y_0,0) \in X \times Y \times Z$ in which the equation $f(x,y)=0$ has a unique solution, which is parameterized by an analytic function $y=g(x) : X \to Y$.
\end{theorem}

The remaining body of this article is devoted to proving that the conditions of the implicit function theorem are satisfied for the above defined operator equation. First, by direct construction of the inverse, we show the linearized problem \ref{LBVP} defines an isomorphism. Required estimates follow from the Hardy inequality. Second, we show that the defining nonlinear operator is analytic. This requires generalizing typical results on composition (superposition) maps in Sobolev spaces to the more exotic functional setting we are working in.


\section{Local solvability by analytic implicit function theorem}

In this section, we prove that the conditions of the implicit function theorem are satisfied for the operator equation \ref{complex operator equation}. First, by direct construction of the inverse, we show the linearized problem \ref{LBVP} defines an isomorphism. Required estimates follow from the Hardy inequality. Second, we show that the defining nonlinear operator is analytic. This requires generalizing typical results on composition (superposition) maps in Sobolev spaces to the more exotic functional setting we are working in.


\subsection{Isomorphism of the linearization}\ \\

The linearization of the nonlinear boundary value problem at reference solution $a(\psi,\theta) = \psi^{1/2}$ is given by:

\begin{definition}[Linear Problem] \label{true linear problem} \
\begin{equation*}
\begin{cases}
\psi^{-1/2} \Bigl[ \psi^2 \pdv[2]{\psi} + 2 \psi \pdv{\psi} + \frac{1}{4} ( I + \pdv[2]{\theta} ) \Bigr] u(\psi,\theta) = f(\psi,\theta) \\
R + ( \frac{p_x-ip_y}{2} )e^{i\theta}  + ( \frac{p_x+ip_y}{2} ) e^{-i\theta} + u(1,\theta)= g(\theta),
\end{cases}
\end{equation*}
to be solved for $R \in \C$, $p = (p_x,p_y) \in \C^2$ and $u(\psi,\theta) \in J_{1/2,\gamma}^{m,\sigma}(\Pi)$, given parameters $f(\psi,\theta) = \widetilde{J}_{0,\gamma}^{m-2,\sigma}(\Pi)$ and $g(\theta) \in X_\sigma^{m-1/2}(\T)$. 
\end{definition}

Since multiplication by $\psi^{1/2}$ defines an isomorphism from $J_{0,\gamma}^{m,\sigma}(\Pi)$ to $J_{1/2,\gamma}^{m,\sigma}(\Pi)$, it is thus equivalent to consider instead the problem
\begin{equation} \label{linear problem}
\begin{cases}
Lu(\psi,\theta) = \Bigl[ \psi^2 \pdv[2]{\psi} + 2 \psi \pdv{\psi} + \frac{1}{4} ( I + \pdv[2]{\theta} ) \Bigr] u(\psi,\theta) = f(\psi,\theta) \\
 R + ( \frac{p_x-ip_y}{2} )e^{i\theta}  + ( \frac{p_x+ip_y}{2} ) e^{-i\theta} + u(1,\theta) = g(\theta),
\end{cases}
\end{equation}
where $f(\psi,\theta)$ is now to be taken in $\widetilde{J}_{1/2,\gamma}^{m-2,\sigma}(\Pi)$.

To show this defines an isomorphism, we break down the proof into four parts. First, the boundedness of the linear map in the above spaces follows immediately from standard results in the Kondratev spaces. Second, we solve the homogeneous problem when $f=0$ and bound the solution by the boundary data. This requires a restriction on the permissible values of $\gamma$. Finally, we solve the inhomogeneous problem on the component spaces that make up $J_{1/2, \gamma}^{m, \sigma}(\Pi)$. The linear problem on the leading term component reduces to an ODE in $\theta$, solved by Fourier inversion. The bulk of the effort is then dedicated to solving the linear problem on the remainder component term in $K_{1/2 + \gamma}^{m,\sigma}(\Pi)$ and establishing the required bounds. Expanding to a Fourier series in $\theta$ yields a sequence of second order ODEs in $\psi$. Each corresponding second order differential operator can be factored into the product of two first order operators. Their inverses, which can be computed explicitly, are operators taking weighted averages. The main tool to establish their boundedness is the Hardy inequality (\cite{BoKo}), which bounds the $L^2$ norm of the weighted average of a function by the $L^2$ norm of said function:

\begin{theorem}[Hardy Inequality]\label{Hardy}\ \\
\begin{itemize}
\item If $\alpha < 1/2$,
\begin{equation*}
 \Bigl\lVert y^{\alpha-1} \int_0^y x^{-\alpha}f(x) \dd{x} \Bigr\rVert_{L^2[0,1]} \leq \frac{1}{\frac{1}{2}-\alpha} \lVert f \rVert_{L^2[0,1]}.
\end{equation*} 
\item If $\alpha > 1/2$,
\begin{equation*}
 \Bigl\lVert y^{\alpha-1} \int_y^1 x^{-\alpha}f(x) \dd{x} \Bigr\rVert_{L^2[0,1]} \leq \frac{1}{\alpha-\frac{1}{2}} \lVert f \rVert_{L^2[0,1]}.
\end{equation*} 
\end{itemize}
\end{theorem}

Let us now proceed, starting with the boundedness of the linear operator.

\begin{proposition}\label{boundedness}\ \\
The linearized problem \ref{linear problem} is bounded in
\begin{equation*}
\C^3 \times J_{1/2, \gamma}^{m, \sigma}(\Pi) \to J_{1/2, \gamma}^{m-2, \sigma}(\Pi) \times X_\sigma^{m-1/2}(\T).
\end{equation*}
\end{proposition}

\begin{proof}
The operators $\psi^2 \partial_\psi^2$, $\psi \partial_\psi$ and $\partial_\theta^2$ are bounded from $J_{1/2, \gamma}^{m, \sigma}(\Pi)$ to $J_{1/2, \gamma}^{m-2, \sigma}(\Pi)$. This follows immediately from definition of $J_{1/2, \gamma}^{m, \sigma}(\Pi)$ and boundedness of these maps in $K_{\lambda + \gamma}^m(\Pi)$. Thus $L$ is bounded in these spaces. Next, the trace map $u(\cdot, \cdot) \to u(1,\theta)$ is bounded from $J_{1/2,\gamma}^{m,\sigma}(\Pi)$ to $X_\sigma^{m-1/2}(\T)$. Using the Fourier series representation of the norm of $X_\sigma^{m-1/2}(\T)$, we get $\lVert R \rVert_{X_{\sigma}^{m-1/2}(\T)} = \lvert R \rvert$ and
$\lVert (p_x \mp ip_y)e^{\pm i \theta} \rVert_{X_{\sigma}^{m-1/2}(\T)} \leq C \left( \lvert p_x \rvert + \lvert p_x \rvert \right)$. Putting it all together gives the bound
\begin{equation*}
\lVert f \rVert_{J_{1/2, \gamma}^{m-2, \sigma}(\Pi)} + \lVert g \rVert_{X_\sigma^{m-1/2}(\T)} \leq \lvert R \rvert + \lvert p \rvert + \lVert u \rVert_{J_{1/2, \gamma}^{m, \sigma}(\Pi)}.
\end{equation*}
\end{proof}

Next, we tackle the homogeneous linear problem.

\begin{proposition}\label{homogeneous}\ \\
Let $1/2 \leq \gamma < 1$. Then the homogeneous problem obtained from \ref{linear problem} by setting $f=0$ is invertible and its solution has bound 
\begin{equation*}
\lvert R \rvert + \lvert p \rvert + \lVert u \rVert_{J_{1/2, \gamma}^{m,\sigma}(\Pi)} \leq C \lVert g \rVert_{X_\sigma^{m-1/2}(\T)}.
\end{equation*}
\end{proposition}

\begin{proof}
Expanding $u(\psi,\theta)$ in a Fourier series in $\theta$ gives the family of 2nd order Cauchy-Euler equations
\begin{equation*}
\left( \psi^2 D^2 + 2\psi D + \frac{1-k^2}{4} \right) \hat{u}_k(\psi) = 0.
\end{equation*}
Solving gives the general homogeneous solution
\begin{equation*}
u(\psi,\theta) = c_0 \psi^{-\frac{1}{2}} + d_0 \psi^{-\frac{1}{2}} \ln(\psi) + \sum_{k \neq 0} \left( c_k \psi^{\frac{-1+\lvert k \rvert}{2}} + d_k \psi^{\frac{-1- \lvert k \rvert}{2}} \right) e^{i k \theta}.
\end{equation*}
The space $J_{1/2, \gamma}^{m, \sigma}(\Pi)$ is the direct sum of a leading term of order $\psi^{1/2}$ and a remainder in $K_{1/2+\gamma}^{m,\sigma}(\Pi)$ of higher order terns. Thus we must discard all terms from the homogeneous solution whose order is less than $\psi^{1/2}$, namely, we must set $c_0$, $c_{\pm 1}$ and all $d_k$ terms to zero. This gives us the homogeneous solution
\begin{equation*}
u(\psi,\theta) = \sum_{\lvert k \rvert \geq 2} c_k \psi^{\frac{-1+\lvert k \rvert}{2}} e^{i k \theta}.
\end{equation*}
Observe that the 0th and 1st order modes are entirely absent from this solution. Their absence is accounted for by the extra degrees of freedom $R$ and $p$, provided in the solution. We split up the solution as follows:
\begin{equation*}
u(\psi, \theta) = \psi^{\frac{1}{2}}\left( c_2e^{2i\theta} + c_{-2}e^{-2i\theta} \right) + \sum_{\lvert k \rvert \geq 3} c_k \psi^{\frac{-1+\lvert k \rvert}{2}} e^{i k \theta}.
\end{equation*}
The first term is of order $\psi^{1/2}$, and its angular contribution is entire, thus certainly in $X_\sigma^m(\T)$. We should now guarantee that the remaining sum belongs to $K_{1/2+\gamma}^{m, \sigma}(\Pi)$ without having to discard any additional modes. We need thus ensure that the lowest order term of the remainder, that is $\psi$, is in $K_{1/2+\gamma}^{m, \sigma}(\Pi)$. This is satisfied when $\gamma < 1$. If we were to allow $\gamma \geq 1$, we would have to drop sufficient additional low order modes from the remainder term, which would render the boundary value problem non-surjective. Also, recall that $J_{1/2, \gamma}^{m, \sigma}(\Pi)$ is only well defined if $\gamma \geq 1/2$. Taking the Fourier series of $g(\theta)$, we can now match the boundary condition. We get:
\begin{equation*}
\begin{cases}
R = \hat{g}_0 \\ 
p_x = \hat{g}_1 + \hat{g}_{-1} \\ 
p_y = i ( \hat{g}_1 - \hat{g}_{-1} ) \\
c_k = \hat{g}_k \text { for } \lvert k \rvert \geq 2.
\end{cases}
\end{equation*}
Finally, we show the bound on this solution. Substituting the above and from definition of $J_{1/2, \gamma}^{m, \sigma}(\Pi)$, we get
\begin{align*}
\lvert R \rvert^2 + \lvert p_x \rvert^2 + \lvert p_y \rvert^2 + \lVert u(\psi,\theta) \rVert_{J_{1/2, \gamma}^{m, \sigma}}^2 & = \lvert \hat{g}_0 \rvert^2 + \lvert \hat{g}_1 + \hat{g}_{-1} \rvert^2 + \lvert \hat{g}_1 -  \hat{g}_{-1} \rvert^2 \\
 + \bigl\lVert \hat{g}_2e^{2i\theta} + \hat{g}_{-2}e^{-2i\theta} \bigr\rVert_{X_\sigma^{m}}^2 & + \Bigr\rVert \sum_{\lvert k \rvert \geq 3} \hat{g}_k \psi^{\frac{-1+\lvert k \rvert}{2}} e^{i k \theta} \Bigr\rVert_{K_{1/2+ \gamma}^{m, \sigma}}^2.
\end{align*}
The last term is bounded as follows:
\begin{align*}
\Bigl\lVert \sum_{\lvert k \rvert \geq 3} \hat{g}_k \psi^{\frac{-1+\lvert k \rvert}{2}} e^{i k \theta} \Bigr\rVert_{K_{1/2+ \gamma}^{m, \sigma}}^2 & = \sum_{p+q=0}^m \sum_{\lvert k \rvert \geq 3} (k^2)^q e^{2 \sigma \lvert k \rvert} \lvert \hat{g}_k \rvert^2 \bigl\lVert \psi^{p-\frac{1}{2}-\gamma} D^p \psi^{\frac{-1+\lvert k \rvert}{2}} \bigr\rVert_{L^2[0,1]}^2 \\ 
& = \sum_{p+q=0}^m \sum_{\lvert k \rvert \geq 3} (k^2)^q e^{2 \sigma \lvert k \rvert} \lvert \hat{g}_k \rvert^2 c_{p,k} \bigl\lVert \psi^{-1-\gamma + \frac{\lvert k \rvert}{2}} \bigr\rVert_{L^2[0,1]}^2 \\ 
& = \sum_{p+q=0}^m \sum_{\lvert k \rvert \geq 3} C_{p,k} \frac{(k^2)^q}{-1-2\gamma + \lvert k \rvert} e^{2 \sigma \lvert k \rvert} \lvert \hat{g}_k \rvert^2  \\ 
& \leq C \sum_{\lvert k \rvert \geq 3} (1+k^2)^{m-1/2} e^{2 \sigma \lvert k \rvert} \lvert \hat{g}_k \rvert^2
\end{align*}
where the third equality follows so long as $\psi^{-1-\gamma+ \frac{\lvert k \rvert}{2}} \in L^2[0,1]$, which is satisfied given $\gamma < 1$ and $\lvert k \rvert \geq 3$. We thus get the required bound
\begin{align*}
\lvert R \rvert^2 + \lvert p_x \rvert^2 + \lvert p_y \rvert^2 + \lVert u(\psi,\theta) \rVert_{J_{1/2, \gamma}^{m, \sigma}(\Pi)}^2 & \leq C \sum_k (1+k^2)^{m-1/2} e^{2 \sigma \lvert k \rvert} \lvert \hat{g}_k \rvert^2 \\
& = C \lVert g(\theta) \rVert_{X_\sigma^{m-1/2}(\T)}.
\end{align*}
\end{proof}

We now tackle the inhomogeneous problem. We write $u(\psi,\theta) = v(\theta)\psi^{1/2}+w(\psi,\theta)$ and $f(\psi,\theta) = \xi(\theta) \psi^{1/2}+ \eta(\psi,\theta)$, where $v \in X_\sigma^m(\T)$, $w \in K_{1/2+\gamma}^{m,\sigma}(\Pi)$, $\xi \in X_\sigma^{m-2}(\T)$ and $\eta \in K_{1/2+\gamma}^{m-2,\sigma}(\Pi)$. We can consider the linear problem on components $v(\theta)\psi^{1/2}$ and $w(\psi,\theta)$ separately, so long as we are careful to distribute the boundary condition modes carefully. Let us start with the first component. We have just seen that $v$ can only account for the $\lvert k \rvert = 2$ modes of the boundary condition. Let us then consider the problem
\begin{equation} \label{first component}
\begin{cases}
L\left(v(\theta)\psi^{1/2}\right) = \xi(\theta)\psi^{1/2} \\
\hat{v}_{\pm 2} = 0.
\end{cases}
\end{equation}

\begin{proposition}\label{first inhomogeneous}\ \\
The linear problem \ref{first component} is invertible between $v(\theta) \in X_\sigma^m(\T)$ and \\ $\xi(\theta) \in \widetilde{X}_\sigma^{m-2}(\T) = \left\{ \xi \in X_\sigma^{m-2}(\T) \colon \int_\T \xi(\theta)e^{\pm 2i\theta} = 0 \right\}$ and its solution has bound $ \lVert v \rVert_{X_\sigma^m(\T)} \leq C \lVert \xi \rVert_{X_\sigma^{m-2}(\T)} $.
\end{proposition}

\begin{proof}
A direct computation shows $L(v\psi^{1/2}) = \psi^{1/2}\left(v(\theta) + \frac{D^2v(\theta)}{4}\right)$. Taking the Fourier series, we get the family of algebraic equations $\left(1-k^2/4 \right) \hat{v}_k = \hat{\xi}_k$. For $\lvert k \rvert = 2$, the left side vanishes. We get the solution
\begin{equation*}
\begin{cases}
\hat{v}_k = (1-k^2/4)^{-1}\hat{\xi}_k, \text{ for } \lvert k \rvert \neq 2,\\
\hat{v}_{\pm 2} = 0
\end{cases}
\end{equation*}
and deduce $L$ is not surjective onto $X_\sigma^{m-2}(\T)$ but rather onto $\widetilde{X}_\sigma^{m-2}(\T)$. To establish the boundedness of this inverse, we have
\begin{align*}
\lVert v \rVert_{X_\sigma^m}^2 & = \sum_k (1+k^2)^m e^{2\sigma \lvert k \rvert} \lvert \hat{v}_k \rvert^2 \\ 
& = \sum_{\lvert k \rvert \neq 2} (1+k^2)^m e^{2\sigma \lvert k \rvert} \frac{\lvert \hat{\xi}_k \rvert^2}{(1-k^2/4)^2} \\ 
& \leq C\sum_k (1+k^2)^{m-2} e^{2\sigma \lvert k \rvert} \lvert \hat{\xi}_k \rvert^2 \\
& \leq C \lVert \xi \rVert_{X_\sigma^{m-2}}^2.
\end{align*}
\end{proof}

It remains to solve the inhomogeneous problem on the second component. We must be careful to distribute the boundary conditions correctly. We have seen that the $k=0$ mode of the boundary condition is controlled for by $R$. Next, the $ \lvert k \rvert = 1$ modes are controlled for by $p$. Finally, the $ \lvert k \rvert = 2$ modes are controlled for by the first component of $u$. Thus we should expect that only the remaining $\lvert k \rvert \geq 3$ modes are controlled by the second component. The inhomogeneous problem then is: given $\eta(\psi,\theta) \in K_{1/2 + \gamma}^{m-2,\sigma}(\Pi)$, solve the following equation for $w(\psi,\theta) \in K_{1/2+ \gamma}^{m,\sigma}(\Pi)$:
\begin{equation} \label{second component}
\begin{cases}
Lw(\psi,\theta) = \eta(\psi,\theta) \\
\hat{w}_k(1) = \int_\T w(1,\theta) e^{- i k \theta} \dd{\theta}  = 0 \text { for } \lvert k \rvert \geq 3.
\end{cases}
\end{equation}

\begin{proposition}\label{second inhomogeneous} \ \\
Let $1/2 < \gamma < 1$. The inhomogeneous problem \ref{second component} is invertible with bound 
\begin{equation*}
\lVert w \rVert_{K_{1/2+\gamma}^{m,\sigma}(\Pi)} \leq  C\lVert \eta \rVert_{K_{1/2+\gamma}^{m-2,\sigma}(\Pi)}.
\end{equation*}
\end{proposition}

\begin{proof}
Expanding in a Fourier series gives the family of ODEs
\begin{equation*}
\begin{cases}
L_k \hat{w}_k(\psi) = \hat{\eta}_k(\psi) \\
\hat{w}_k(1) = 0 \text { for } \lvert k \rvert \geq 3,
\end{cases}
\end{equation*}
where 
\begin{equation*}
L_k = \psi^2 D^2 + 2 \psi D + \frac{1-k^2}{4} I.
\end{equation*}
We can factor $L_k$ into the product of two 1st order operators as follows:
\begin{equation*}
L_k = \left(\psi D + \frac{1+ \lvert k \rvert}{2} I \right) \left(\psi D + \frac{1- \lvert k \rvert}{2} I \right) = L_k^{+} \cdot L_k^{-}.
\end{equation*}
Next we can rewrite these operators as
\begin{align*}
L_k^{+} = \psi D + \frac{1+ \lvert k \rvert}{2} I = \psi^{1-\frac{1+ \lvert k \rvert}{2}} D \left( \psi^{\frac{1+ \lvert k \rvert}{2}} I \right), \\
L_k^{-} = \psi D + \frac{1- \lvert k \rvert}{2} I = \psi^{1-\frac{1- \lvert k \rvert}{2}} D \left( \psi^{\frac{1- \lvert k \rvert}{2}} I \right).
\end{align*}
We can accordingly define a factorization of $L$ by
\begin{equation*}
L = L_{+} \cdot L_{-},
\end{equation*}
where
\begin{align*}
L_{\pm} w(\psi,\theta) = \sum_k L_k^{\pm} \hat{w}_k(\psi) e^{i k \theta}.
\end{align*}
Notice each of these operators is bounded from $K_{1/2 + \gamma}^{m,\sigma}(\Pi)$ to $K_{1/2 + \gamma}^{m-1,\sigma}(\Pi)$. Using the boundary conditions, we can explicitly invert $L_k^{\pm}$ to construct $ L_{\pm}^{-1}$ and show that it is bounded from $K_{1/2+\gamma}^{m-1, \sigma}(\Pi)$  to $K_{1/2+\gamma}^{m, \sigma}(\Pi)$. Since $L^{-1}$ is the composition of $L_{-}^{-1}$ and $L_{+}^{-1}$, it follows that it is bounded from $K_{1/2+\gamma}^{m-2 ,\sigma}(\Pi)$ to $K_{1/2+\gamma}^{m ,\sigma}(\Pi)$.

Let us proceed now with inverting $L_k^{+}$ and $L_k^{-}$. They take the general form of operator
\begin{equation*}
A_{\lambda_k} = \psi D + \lambda_k I = \psi^{1-\lambda_k} D \left( \psi^{\lambda_k} I \right),
\end{equation*}
where in our case $\lambda_k = \frac{1 \pm \lvert k \rvert}{2}$. Writing the equation 
\begin{equation*}
A_{\lambda_k} w_k(\psi) = \psi^{1-\lambda_k} D \left( \psi^{\lambda_k} w_k(\psi) \right) = \eta_k(\psi),
\end{equation*}
we can solve by direct integration to get
\begin{equation*}
w_k(\psi) = A_{\lambda_k}^{-1}\eta_k(\psi) = \psi^{-\lambda_k} \int_0^{\psi} t^{\lambda_k-1} \eta_k(t) \dd{t} + c\psi^{-\lambda_k},
\end{equation*}
or equivalently
\begin{equation*}
w_k(\psi) = A_{\lambda_k}^{-1}\eta_k(\psi) = -\psi^{-\lambda_k} \int_{\psi}^{1} t^{\lambda_k-1} \eta_k(t) \dd{t} + d\psi^{-\lambda_k}.
\end{equation*}
Next, we should check for which $\lambda_k = \frac{1 \pm \lvert k \rvert}{2}$ does the $\psi^{-\lambda_k}$ term belong to $K_{1/2+\gamma}^{m,\sigma}$, keeping in mind we have the restriction $ 1/2 \leq \gamma < 1$ from the homogeneous problem. This occurs only for $L_k^{-}$ when $\lvert k \rvert \geq 3$, that is, when $\lambda_k = \frac{1 - \lvert k \rvert}{2}$ and $\lvert k \rvert \geq 3$. In this case, we need the boundary condition to find the inverse. Otherwise, we must set constants $c$ or $d$ to zero and no boundary condition is available. Let us then write the inverses as follows.
\begin{equation} \label{minus inverse}
\left( L_k^{-} \right)^{-1} \hat{\eta}_k(\psi) = 
\begin{cases}
\psi^{-\frac{1 - \lvert k \rvert}{2}} \int_0^\psi t^{\frac{-1-\lvert k \rvert}{2}} \hat{\eta}_k(t) \dd{t} \text { \quad for } \lvert k \rvert <  3 \\
-\psi^{-\frac{1 - \lvert k \rvert}{2}} \int_{\psi}^1 t^{\frac{-1-\lvert k \rvert}{2}} \hat{\eta}_k(t) \dd{t} \text { \quad for } \lvert k \rvert \geq 3
\end{cases}
\end{equation}
\begin{equation}\label{plus inverse}
\left( L_k^{+} \right)^{-1} \hat{\eta}_k(\psi) = \psi^{-\frac{1 + \lvert k \rvert}{2}} \int_0^\psi t^{\frac{-1+\lvert k \rvert}{2}} \hat{\eta}_k(t) \dd{t} \text { \quad for all } k. 
\end{equation}
Notice that the inverses above are operators that take the weighted average of a Fourier mode of $\eta$ from $0$ to $\psi$ or from $\psi$ to $1$. The choice we made is in anticipation of using the Hardy inequality \ref{Hardy}. Also notice the boundary conditions available were each used precisely once. We have thus constructed explicitly the inverse of $L$, which can be written as
\begin{equation*}
L^{-1} \eta(\psi,\theta) = \sum_k \left(L_k^{-} \right)^{-1} \cdot \left( L_k^{+}\right)^{-1} \hat{\eta}_k(\psi) e^{i k \theta}
\end{equation*}
and can be described as the Fourier series of the composition of two varying weighted averages of the Fourier modes of $\eta$.

Now, we must demonstrate the boundedness from $K_{1/2+ \gamma}^{m-1, \sigma}(\Pi)$ to $K_{1/2+ \gamma}^{m, \sigma}(\Pi)$ of operators
\begin{equation*}
L_{\pm}^{-1} \eta(\psi,\theta) = \sum_k \left( L_k^{\pm} \right)^{-1} \hat{\eta}_k(\psi) e^{i k \theta}.
\end{equation*}
That is, we are looking to establish the bound
\begin{equation*}
 \lVert L_{\pm}^{-1} \eta(\psi,\theta) \rVert_{K_{1/2+\gamma}^{m, \sigma}} \leq C  \lVert \eta(\psi,\theta) \rVert_{K_{1/2+\gamma}^{m-1, \sigma}}
 \end{equation*}
 given norm
\begin{equation}\label{norm}
\lVert w(\psi, \theta) \rVert_{K_{1/2+\gamma}^{m,\sigma}(\Pi)}^2 = \sum_k \sum_{p+q=0}^m (k^2)^q e^{2 \sigma \lvert k \rvert} \bigl\lVert \psi^{p-1/2-\gamma} D^p \hat{w}_k(\psi) \bigr\rVert_{L^2[0,1]}^2.
\end{equation}
Having already constructed $\hat{w}_k(\psi)$ in \ref{minus inverse} and \ref{plus inverse}, we should next find an expression for its derivatives $D^p \hat{w}_k(\psi)$. Let us again write
\begin{equation*}
A_{\lambda_k} \hat{w}_k(\psi) = \left( \psi D + \lambda_k I \right) \hat{w}_k(\psi) = \hat{\eta}_k(\psi).
\end{equation*}
Rearranging, gives
\begin{equation*}
D \hat{w}_k(\psi) = \frac{1}{\psi} \hat{\eta}_k(\psi) - \frac{\lambda_k}{\psi} \hat{w}_k(\psi).
\end{equation*}
Continued differentiation and substitution yields the expression
\begin{equation}
D^p \hat{w}_k = \sum_{n=1}^p (-1)^{n+1} \frac{(\lambda_k+p-1)!}{(\lambda_k+p-n)!} \cdot \frac{1}{\psi^n} \cdot D^{p-n} \hat{\eta}_k + (-1)^p (\lambda_k+p-1)! \frac{\hat{w}_k}{\psi^p},
\end{equation}
where we use the factorial sign $!$ in a modified sense to mean
\begin{equation} \label{factorial}
\begin{cases}
(\lambda + l)! =  (\lambda + l)(\lambda + l-1) \cdot \cdot \cdot (\lambda+1)(\lambda) \text { \quad for } l \in \N_0, \\
(\lambda-1)! = 1.
\end{cases}
\end{equation}
Substituting this expression into \ref{norm} and using triangle inequality for the summations, we get
\begin{align*}
& \lVert w(\psi, \theta) \rVert_{K_{1/2+\gamma}^{m,\sigma}}^2 \\ & \leq C \sum_k \sum_{\substack{p+q=1\\ p \geq 1}}^m (k^2)^q e^{2 \sigma \lvert k \rvert} \sum_{n=1}^p \left( \frac{(\lambda_k+p-1)!}{(\lambda_k+p-n)!} \right)^2 \bigl\lVert \psi^{p-n-\frac{1}{2}-\gamma} D^{p-n} \hat{\eta}_k(\psi) \bigr\rVert_{L^2[0,1]}^2 \\
& + C \sum_k \sum_{p+q=0}^m (k^2)^q e^{2 \sigma \lvert k \rvert} \bigl( (\lambda_k + p -1)! \bigr)^{2} \bigl\lVert \psi^{-\frac{1}{2}-\gamma} \hat{w}_k(\psi) \bigr\rVert_{L^2[0,1]}^2.
\end{align*}
From our modified factorial \ref{factorial} and that $\lambda_k = \frac{1\pm \lvert k \rvert}{2}$, we have that
\begin{equation*}
(\lambda_k+p-1)! \sim \lvert k \rvert^{p} \text{  as } \lvert k \rvert \to \pm \infty
\end{equation*}
and
\begin{equation*}
\frac{(\lambda_k+p-1)!}{(\lambda_k+p-n)!} \sim \lvert k \rvert^{n-1} \text{ as } \lvert k \rvert \to \pm \infty.
\end{equation*}
Applying this to the inequality above gives
\begin{align*}
& \lVert w(\psi, \theta) \rVert_{K_{1/2+\gamma}^{m,\sigma}}^2 \\ & \leq C \sum_k \sum_{\substack{p+q=1\\ p \geq 1}}^m \sum_{n=1}^p (k^2)^{n-1+q} e^{2 \sigma \lvert k \rvert} \bigl\lVert \psi^{p-n-\frac{1}{2}-\gamma} D^{p-n} \hat{\eta}_k(\psi) \bigr\rVert_{L^2[0,1]}^2 \\
& + C \sum_k \sum_{p+q=0}^m (k^2)^{p+q} e^{2 \sigma \lvert k \rvert} \bigl\lVert \psi^{-\frac{1}{2}-\gamma} \hat{w}_k(\psi) \bigr\rVert_{L^2[0,1]}^2.
\end{align*}
Let us split the right hand side into two terms, with
\begin{equation*}
A = \sum_k \sum_{\substack{p+q=1\\ p \geq 1}}^m \sum_{n=1}^p (k^2)^{n-1+q} e^{2 \sigma \lvert k \rvert} \bigl\lVert \psi^{p-n-\frac{1}{2}-\gamma} D^{p-n} \hat{\eta}_k(\psi) \bigr\rVert_{L^2[0,1]}^2
\end{equation*}
and
\begin{equation*}
B = \sum_k \sum_{p+q=0}^m (k^2)^{p+q} e^{2 \sigma \lvert k \rvert} \bigl\lVert \psi^{-\frac{1}{2}-\gamma} \hat{w}_k(\psi) \bigr\rVert_{L^2[0,1]}^2.
\end{equation*}
In term $A$, setting $q' = n-1+q$ and $p' = p-n$, then $p'+q' = p+q-1$ ranges from $0$ to $m-1$. We immediately get
\begin{equation*}
A \leq C  \sum_k \sum_{p'+q'=0}^{m-1} (k^2)^{q'} e^{2 \sigma \lvert k \rvert} \bigl\lVert \psi^{p'-\frac{1}{2}-\gamma} D^{p'} \hat{\eta}_k(\psi) \bigr\rVert_{L^2[0,1]}^2 = C \lVert \eta(\psi,\theta) \rVert_{K_{1/2+\gamma}^{m-1, \sigma}}.
\end{equation*}

Now we turn to bounding $B$, starting with the factor $\bigl\lVert \psi^{-\frac{1}{2}-\gamma} \hat{w}_k(\psi) \bigr\rVert_{L^2[0,1]}$ in the summation. To achieve this we will use now the Hardy inequality. Recall, $\hat{w}_k(\psi)$ is given by \ref{plus inverse} and \ref{minus inverse}, depending on if we are inverting $L_k^{+}$ or $L_k^{-}$. This gives us three separate cases of expressions for $\hat{w}_k(\psi)$. 

Let us start with the case when $\hat{w}_k(\psi) = \left( L_k^{+} \right)^{-1} \hat{\eta}_k(\psi)$, and $k$ is arbitrary. In this case,
\begin{equation*}
\hat{w}_k(\psi) = \psi^{-\frac{1+\lvert k \rvert}{2}}\int_0^\psi t^{\frac{-1+\lvert k \rvert}{2}} \hat{\eta}_k(t) \dd{t}.
\end{equation*}
Setting $\alpha-1 = - \frac{1}{2} - \gamma - \frac{1 + \lvert k \rvert}{2}$, we get $\alpha = -\gamma - \frac{\lvert k \rvert}{2} \leq -\gamma < 1/2$ for all $k$, as long as $\gamma > -1/2$. Also set $t^{\frac{-1+\lvert k \rvert}{2}} \hat{\eta}_k(t) = t^{-\alpha}\zeta_k(t)$. Applying Hardy's inequality for $\alpha < 1/2$, we get
\begin{align*}
\bigl\lVert \psi^{-\frac{1}{2}-\gamma} \hat{w}_k(\psi) \bigr\rVert_{L^2[0,1]} & = \Bigl\lVert \psi^{-\frac{1}{2}-\gamma-\frac{1+\lvert k \rvert}{2}} \int_0^\psi t^{\frac{-1+\lvert k \rvert}{2}} \hat{\eta}_k(t) \dd{t} \Bigr\rVert_{L^2[0,1]} \\
& =  \Bigl\lVert \psi^{\alpha-1} \int_0^\psi t^{-\alpha} \zeta_k(t) \dd{t} \Bigr\rVert_{L^2[0,1]} \\
& \leq \frac{1}{\frac{1}{2}-\alpha} \left\lVert \zeta_k(\psi) \right\rVert_{L^2[0,1]} \\
& \leq \frac{1}{\frac{1}{2} + \gamma + \frac{\lvert k\rvert}{2}} \bigl\lVert \psi^{-\gamma-\frac{1}{2}} \hat{\eta}_k(\psi) \bigr\rVert_{L^2[0,1]}.
\end{align*}

Next, we consider the case when $\hat{w}_k(\psi) = \left( L_k^{-} \right)^{-1} \hat{\eta}_k(\psi)$ and $ \lvert k \rvert < 3$. In this case,
\begin{equation*}
\hat{w}_k(\psi) = \psi^{-\frac{1-\lvert k \rvert}{2}}\int_0^\psi t^{\frac{-1-\lvert k \rvert}{2}} \hat{\eta}_k(t) \dd{t}.
\end{equation*}
Setting $\alpha-1 = - \frac{1}{2} - \gamma - \frac{1 - \lvert k \rvert}{2}$, we get $\alpha = -\gamma + \frac{\lvert k \rvert}{2} < 1/2$ for $\lvert k \rvert \leq 2$, only as long as $\gamma > 1/2$. Also set $t^{\frac{-1-\lvert k \rvert}{2}} \hat{\eta}_k(t) = t^{-\alpha}\zeta_k(t)$. Applying Hardy's inequality for $\alpha < 1/2$, we get
\begin{align*}
\bigl\lVert \psi^{-\frac{1}{2}-\gamma} \hat{w}_k(\psi) \bigr\rVert_{L^2[0,1]} & = \Bigl\lVert \psi^{-\frac{1}{2}-\gamma-\frac{1- \lvert k \rvert}{2}} \int_0^\psi t^{\frac{-1-\lvert k \rvert}{2}} \hat{\eta}_k(t) \dd{t} \Bigr\rVert_{L^2[0,1]} \\
& =  \Bigl\lVert \psi^{\alpha-1} \int_0^\psi t^{-\alpha} \zeta_k(t) \dd{t} \Bigr\rVert_{L^2[0,1]} \\
& \leq \frac{1}{\frac{1}{2}-\alpha} \left\lVert \zeta_k(\psi) \right\rVert_{L^2[0,1]} \\
& \leq \frac{1}{\frac{1}{2} + \gamma - \frac{\lvert k\rvert}{2}} \bigl\lVert \psi^{-\gamma-\frac{1}{2}} \hat{\eta}_k(\psi) \bigr\rVert_{L^2[0,1]}.
\end{align*}

The third case occurs when $\hat{w}_k(\psi) = \left( L_k^{-} \right)^{-1} \hat{\eta}_k(\psi)$ and $ \lvert k \rvert \geq3$. In this case,
\begin{equation*}
\hat{w}_k(\psi) = -\psi^{-\frac{1-\lvert k \rvert}{2}}\int_\psi^1 t^{\frac{-1-\lvert k \rvert}{2}} \hat{\eta}_k(t) \dd{t}.
\end{equation*}
Setting $\alpha-1 = - \frac{1}{2} - \gamma - \frac{1 - \lvert k \rvert}{2}$, we get $\alpha = -\gamma + \frac{\lvert k \rvert}{2} > 1/2$ for $\lvert k \rvert \geq 3$,  as long as $\gamma < 1$. Also set $t^{\frac{-1-\lvert k \rvert}{2}} \hat{\eta}_k(t) = t^{-\alpha}\zeta_k(t)$. Applying Hardy's inequality now for $\alpha > 1/2$, we get
\begin{align*}
\bigl\lVert \psi^{-\frac{1}{2}-\gamma} \hat{w}_k(\psi) \bigr\rVert_{L^2[0,1]} & = \Bigl\lVert \psi^{-\frac{1}{2}-\gamma-\frac{1- \lvert k \rvert}{2}} \int_\psi^1 t^{\frac{-1-\lvert k \rvert}{2}} \hat{\eta}_k(t) \dd{t} \Bigr\rVert_{L^2[0,1]} \\
& =  \Bigl\lVert \psi^{\alpha-1} \int_0^\psi t^{-\alpha} \zeta_k(t) \dd{t} \Bigl\rVert_{L^2[0,1]} \\
& \leq \frac{1}{\alpha - \frac{1}{2}} \left\lVert \zeta_k(\psi) \right\rVert_{L^2[0,1]} \\
& \leq \frac{1}{-\frac{1}{2} - \gamma + \frac{\lvert k\rvert}{2}} \bigl\lVert \psi^{-\gamma-\frac{1}{2}} \hat{\eta}_k(\psi) \bigr\rVert_{L^2[0,1]}.
\end{align*}

We have thus found that in each of the three cases, we get the bound
\begin{equation*}
\bigl\lVert \psi^{-\frac{1}{2}-\gamma} \hat{w}_k(\psi) \bigr\rVert_{L^2[0,1]} \leq C_{\gamma, k}  \bigl\lVert \psi^{-\frac{1}{2} - \gamma} \hat{\eta}_k(\psi) \bigr\rVert_{L^2[0,1]}.
\end{equation*}
The crucial detail is the additional strict restriction to $\gamma > 1/2$, which avoids the critical case of the Hardy inequality when $\alpha = 1/2$. This ensures the constant $C_{\gamma, k}$ above is bounded for all $k$. Note this constant decays like $2/{\lvert k \rvert}$.

Returning to $B$ and applying the above, we get the bound
\begin{align*}
B & = \sum_k \sum_{p+q=0}^m (k^2)^{p+q} e^{2 \sigma \lvert k \rvert} \bigl\lVert \psi^{-\frac{1}{2}-\gamma} \hat{w}_k(\psi) \bigr\rVert_{L^2[0,1]}^2 \\
& \leq C \sum_k \sum_{q'=0}^{m-1} (k^2)^{q'} e^{2 \sigma \lvert k \rvert} \bigl\lVert \psi^{-\frac{1}{2}-\gamma} \hat{\eta}_k(\psi) \bigr\rVert_{L^2[0,1]}^2 \\
& \leq C \left\lVert \eta(\psi,\theta) \right\rVert_{K_{1/2 + \gamma}^{m-1, \sigma}(\Pi)}.
\end{align*}

From $A$ and $B$, we get
\begin{equation*}
\left\lVert L_{\pm}^{-1} \eta(\psi,\theta) \right\rVert_{K_{1/2+\gamma}^{m ,\sigma}} \leq C \left\lVert \eta(\psi,\theta) \right\rVert_{K_{1/2+\gamma}^{m-1 ,\sigma}}
\end{equation*}
and thus
\begin{equation*}
 \lVert w(\psi,\theta) \rVert_{K_{1/2+\gamma}^{m ,\sigma}} = \left\lVert L_{-}^{-1} \cdot L_{+}^{-1} \eta(\psi,\theta) \right\rVert_{K_{1/2+\gamma}^{m ,\sigma}} \leq C \left\lVert \eta(\psi,\theta) \right\rVert_{K_{1/2+\gamma}^{m-2 ,\sigma}}.
\end{equation*}
\end{proof}

We now have the ingredients to prove the main result of this section.
\begin{theorem}
For $1/2 < \gamma < 1$, the linear problem \ref{true linear problem} defines an isomorphism 
\begin{equation*}
\C^3 \times J_{1/2, \gamma}^{m, \sigma}(\Pi) \to \widetilde{J}_{0, \gamma}^{m-2, \sigma}(\Pi) \times X_{\sigma}^{m-1/2}(\T).
\end{equation*}
\end{theorem}

\begin{proof}
By proposition \ref{boundedness} and the fact that multiplication by $\psi^{-1/2}$ defines an isomorphism from $J_{1/2, \gamma}^{m-2,\sigma}(\Pi)$ to $J_{0, \gamma}^{m-2,\sigma}(\Pi)$, the linear map is bounded in the above spaces.

To construct and bound the inverse, we must be careful to match the boundary conditions correctly. Let $u_g = \psi^{1/2}v(\theta) + w(\psi,\theta)$, where $v(\theta)$ is solution to \ref{first component} and $w(\psi,\theta)$ is solution to \ref{second component}. Let $u_h = \sum_{\lvert k \rvert \geq 2} c_k \psi^{\frac{-1+\lvert k \rvert}{2}} e^{ik\theta}$ be the homogeneous solution with coefficients $c_k$ to be determined. Then the full solution is
\begin{align*}
u(\psi,\theta) & = u_g + u_h \\
& = \psi^{\frac{1}{2}} \left( \hat{v}(\theta) +c_2e^{2i\theta}+c_{-2}e^{-2i\theta} \right) + w(\psi,\theta) + \sum_{\lvert k \rvert \geq 3}c_k \psi^{\frac{-1+\lvert k \rvert}{2}} e^{ik\theta} \\ 
& = \psi^{\frac{1}{2}} \left( \sum_k \hat{v}_k e^{ik\theta} + \sum_{\lvert k \rvert = 2} c_k e^{ik\theta} \right)+  \sum_k \hat{w}_k(\psi)e^{ik\theta} +\sum_{\lvert k \rvert \geq 3}c_k \psi^{\frac{-1+\lvert k \rvert}{2}} e^{ik\theta}.
\end{align*}
Now we must match the boundary condition
\begin{equation*}
R + ( \frac{p_x-ip_y}{2} )e^{i\theta}  + ( \frac{p_x+ip_y}{2} ) e^{-i\theta} + u(1,\theta) = g(\theta).
\end{equation*}
Keeping in mind that $\hat{v}_{\pm 2} = 0$, and $\hat{w}_k(1) = 0$ for $\lvert k \rvert \geq 3$, the boundary condition yields
\begin{align*}
& R + ( \frac{p_x-ip_y}{2} )e^{i\theta}  + ( \frac{p_x+ip_y}{2} ) e^{-i\theta} + \sum_{\lvert k \rvert \neq 2} \hat{v}_ke^{ik\theta} + \sum_{\lvert k \rvert \leq 2} \hat{w}_k(1)e^{ik\theta} + \sum_{\lvert k \rvert \geq 2} c_k e^{ik\theta} \\ & = \sum_k \hat{g}_k e^{ik\theta}.
\end{align*} 
This gives us the following set of equations on the Fourier modes
\begin{equation*}
\begin{cases}
R + \hat{v}_0 + \hat{w}_0(1) = \hat{g}_0 \\
\frac{p_x \mp ip_y}{2} + \hat{v}_{\pm 1} + \hat{w}_{\pm 1}(1) = \hat{g}_{\pm 1} \\
\hat{w}_{\pm 2}(1) + c_{\pm 2} = \hat{g}_{\pm 2}\\
\hat{v}_k + c_k = \hat{g}_k, \text { for } \lvert k \rvert \geq 3.
\end{cases}
\end{equation*}
Solving for $R$, $p_x$, $p_y$ and $c_k$, we get
\begin{equation*}
\begin{cases}
R = \hat{g}_0 -\hat{v}_0 - \hat{w}_0(1) \\
p_x = \bigl(\hat{g}_{1} +\hat{g}_{-1}\bigr) - \bigl(\hat{v}_{1} + \hat{v}_{-1}\bigr) - \bigl(\hat{w}_{1}(1) + \hat{w}_{-1}(1)\bigr) \\
-ip_y = \bigl(\hat{g}_{1} - \hat{g}_{-1}\bigr) - \bigl(\hat{v}_{1} - \hat{v}_{-1}\bigr) - \bigl(\hat{w}_{1}(1) - \hat{w}_{-1}(1)\bigr) \\
c_{\pm 2} = \hat{g}_{\pm 2} - \hat{w}_{\pm 2}(1) \\
c_k = \hat{g}_k - \hat{v}_k, \text { for } \lvert k \rvert \geq 3.
\end{cases}
\end{equation*}
Having now constructed the solution to the boundary value \ref{true linear problem}, we can now establish its boundedness. We start with
\begin{align*}
\lvert R \rvert^2 + \lvert p_x \rvert^2 + \lvert p_y \rvert^2 & \leq C \sum_{\lvert k \rvert \leq 1} \bigl( \lvert \hat{g}_k \rvert^2 + \lvert \hat{v}_k \rvert^2 + \lvert \hat{w}_k(1) \rvert^2 \bigr) \\
& \leq C \left( \lVert g(\theta) \rVert_{X_\sigma^{m-1/2}}^2 + \lVert v(\theta) \rVert_{X_\sigma^m}^2 + \lVert w(1,\theta) \rVert_{X_\sigma^{m-1/2}}^2 \right) \\
& \leq C \left( \lVert g(\theta) \rVert_{X_\sigma^{m-1/2}}^2 + \lVert v(\theta) \rVert_{X_\sigma^m}^2 + \lVert w(\psi,\theta) \rVert_{K_{1/2+\gamma}^{m,\sigma}}^2 \right) \\
& \leq C \left( \lVert g(\theta) \rVert_{X_\sigma^{m-1/2}}^2 + \lVert \xi(\theta) \rVert_{X_\sigma^{m-2}}^2 + \lVert \eta(\psi,\theta) \rVert_{K_{1/2+\gamma}^{m-2,\sigma}}^2 \right) \\
& = C \left( \lVert g(\theta) \rVert_{X_\sigma^{m-1/2}}^2 + \lVert f(\psi,\theta) \rVert_{J_{1/2, \gamma}^{m-2,\sigma}}^2 \right),
\end{align*}
where the third inequality follows from the second by boundedness of restriction to $\psi=1$ from $K_{1/2+\gamma}^{m,\sigma}(\Pi)$ to $X_\sigma^{m-1/2}(\T)$, and the fourth inequality follows from the third by propositions \ref{first inhomogeneous}, \ref{second inhomogeneous}.

Next we bound the leading term of the solution
\begin{align*}
& \bigl\lVert v(\theta) + \sum_{\lvert k \rvert = 2}c_k e^{ik\theta} \bigr\rVert_{X_\sigma^m}^2 \\
& \leq C \lVert v(\theta) \rVert_{X_\sigma^m}^2 + C \sum_{\lvert k \rvert = 2} (1+k^2)^m e^{2\sigma \lvert k \rvert} \left( \lvert \hat{g}_k \rvert^2 + \lvert \hat{w}_k(1) \rvert^2 \right) \\
& \leq C \lVert v(\theta) \rVert_{X_\sigma^m}^2 + C \sum_{\lvert k \rvert = 2} (1+k^2)^{m-1/2} e^{2\sigma \lvert k \rvert} \left( \lvert \hat{g}_k \rvert^2 + \lvert \hat{w}_k(1) \rvert^2 \right) \\
& \leq C \left( \lVert v(\theta) \rVert_{X_\sigma^m}^2 + \lVert g(\theta) \rVert_{X_\sigma^{m-1/2}}^2 + \lVert w(1,\theta) \rVert_{X_\sigma^{m-1/2}}^2 \right) \\
& \leq C \left( \lVert v(\theta) \rVert_{X_\sigma^m}^2 + \lVert g(\theta) \rVert_{X_\sigma^{m-1/2}}^2 + \lVert w(\psi,\theta) \rVert_{K_{1/2+\gamma}^{m, \sigma}}^2 \right) \\
& \leq C \left( \lVert \xi(\theta) \rVert_{X_\sigma^{m-2}}^2 + \lVert g(\theta) \rVert_{X_\sigma^{m-1/2}}^2 + \lVert \eta(\psi,\theta) \rVert_{K_{1/2+\gamma}^{m-2, \sigma}}^2 \right) \\
& \leq C \left( \lVert g(\theta) \rVert_{X_\sigma^{m-1/2}}^2 + \lVert f(\psi,\theta) \rVert_{J_{1/2,\gamma}^{m-2, \sigma}}^2 \right),
\end{align*}
where again we have used the boundedness of the restriction to $\psi=1$ and propositions \ref{first inhomogeneous}, \ref{second inhomogeneous}.

Finally we bound the remainder term of the solution
\begin{align*}
\bigl\lVert w(\psi,\theta) + & \sum_{\lvert k \rvert \geq 3} c_k \psi^{\frac{-1+\lvert k \rvert}{2}}e^{i k\theta} \bigr\rVert_{K_{1/2+\gamma}^{m,\sigma}}^2 \\
& \leq C \lVert w(\psi,\theta) \rVert_{K_{1/2+\gamma}^{m,\sigma}}^2 + C \bigl\lVert \sum_{\lvert k \rvert \geq 3} (\hat{g}_k - \hat{v}_k) \psi^{\frac{-1+\lvert k \rvert}{2}}e^{i k\theta} \bigr\rVert_{K_{1/2+\gamma}^{m,\sigma}}^2 \\
& \leq C \lVert w(\psi,\theta) \rVert_{K_{1/2+\gamma}^{m,\sigma}}^2 + \lVert g(\theta) - v(\theta) \rVert_{X_\sigma^{m-1/2}}^2 \\
& \leq C \lVert w(\psi,\theta) \rVert_{K_{1/2+\gamma}^{m,\sigma}}^2 + \lVert g(\theta) \rVert_{X_\sigma^{m-1/2}}^2+ \lVert v(\theta) \rVert_{X_\sigma^{m-1/2}}^2 \\
& \leq C \lVert w(\psi,\theta) \rVert_{K_{1/2+\gamma}^{m,\sigma}}^2 + \lVert g(\theta) \rVert_{X_\sigma^{m-1/2}}^2+ \lVert v(\theta) \rVert_{X_\sigma^{m}}^2 \\
& \leq C \lVert \eta(\psi,\theta) \rVert_{K_{1/2+\gamma}^{m-2,\sigma}}^2 + \lVert g(\theta) \rVert_{X_\sigma^{m-1/2}}^2+ \lVert \xi(\theta) \rVert_{X_\sigma^{m-2}}^2 \\
& = C \left(  \lVert g(\theta) \rVert_{X_\sigma^{m-1/2}}^2 + \lVert f(\psi,\theta) \rVert_{J_{1/2, \gamma}^{m-2,\sigma}}^2 \right)
\end{align*}
where we have used propositions \ref{homogeneous}, \ref{first inhomogeneous} and \ref{second inhomogeneous}.

We have thus established the bound
\begin{equation*}
\lvert R \rvert^2 + \lvert p \rvert^2 + \lVert u(\psi,\theta) \rVert_{J_{1/2, \gamma}^{m,\sigma}(\Pi)} \leq C \left(  \lVert g(\theta) \rVert_{X_\sigma^{m-1/2}(\T)}^2 + \lVert f(\psi,\theta) \rVert_{J_{1/2, \gamma}^{m-2,\sigma}(\Pi)}^2 \right)
\end{equation*}
for the inverse to \ref{linear problem}.

Finally, since multiplication by $\psi^{-1/2}$ is an isomorphism from $J_{1/2, \gamma}^{m-2,\sigma}(\Pi)$ to $J_{0, \gamma}^{m-2,\sigma}(\Pi)$, we conclude that the linear problem \ref{true linear problem} defines an isomorphism
\begin{equation}
\C^3 \times J_{1/2, \gamma}^{m, \sigma}(\Pi) \to \widetilde{J}_{0, \gamma}^{m-2, \sigma}(\Pi) \times X_\sigma^{m-1/2}(\T)
\end{equation}
for $1/2 < \gamma < 1$.
\end{proof}

To conclude, let us summarize our findings. We have shown that the linearization to the nonlinear problem \ref{NLBVP} at solution $\psi^{1/2}$ defines an isomorphism in our spaces when $1/2 < \gamma < 1$. Recall that $J_{\lambda+\gamma}^{m,\sigma}(\Pi)$ is well defined for $\gamma \geq 1/2$ and that $\gamma$ quantifies the gap between the leading term asymptotics of order $\psi^\lambda$ and the order of asymptotics of the remainder term. The restriction $\gamma < 1$ ensures that this gap is small enough that we have sufficient low order terms to match the boundary condition. The restriction $\gamma > 1/2$ is needed to apply the Hardy inequality to establish the boundedness of the inverse. Recall that though $J_{1/2+\gamma}^{m,\sigma}(\Pi)$ is well defined for $\gamma \geq 1/2$, it is precisely when $\gamma > 1/2$ that this space embeds into continuous functions. Since continuity is certainly not an unreasonable expectation of our stationary flows, we need not view this restriction on the lower bound of $\gamma$ as a limitation of our result. Finally, we have seen how the additional degrees of freedom in the solution provided by $R$ and $p$ are crucial for the surjectivity of the linear problem. They accommodate for the $\lvert k \rvert \leq 1$ Fourier modes of the boundary perturbation, corresponding to infinitesimal dilations and translations of the solution.


\newpage

\subsection{Analyticity of the nonlinear operator}\ \\

We prove the following result.

\begin{theorem}\ \\
Let $m>3$ and $\gamma > 1/2$. For any $\tau > \sigma > 0$, there exists a neighbourhood of $(F,b,R,p,a) = (4,1,1,0, \psi^{1/2})$ in which the map
\begin{gather*}
(F,b,R,p,a) \to  \big(\Xi(a) - F, B \big)\\
 J_{0,\gamma}^{m-2}[0,1] \times \mathrm{H}(\T_\tau) \times \C^3 \times J_{1/2,\gamma}^{m,\sigma}(\Pi) \to \widetilde{J}_{0,\gamma}^{m-2,\sigma}(\Pi) \times X_\sigma^{m-1/2}(\T)
\end{gather*}
is complex analytic.
\end{theorem}

We split the proof in two main parts: one for the differential operator and one for the boundary operator. We will find that these reduce to the study of superposition operators in spaces $J_{0, \gamma}^{m, \sigma}(\Pi)$ and $X_\sigma^{m-1/2}(\T)$ respectively. Such maps turn out to be analytic precisely when these spaces are algebras. We start with a standard result on superposition operators acting on $H^m(\T)$ when $m > 1/2$, that is when $H^m(\T) \in C(\T)$.

\begin{proposition}\label{superposition sobolev} \ \\
Let $f \in C^{m}$ on a domain containing image of function $u$. Then $u \to f(u) : H^m(\T) \to H^m(\T)$ is well defined and continuous for $m > 1/2$. Furthermore,  if $f \in C^{m+1}$, then this map is $C^1$.
\end{proposition}

\begin{proof}
The $p$-th order derivative in $x$ of composition $f\big(u(x)\big)$ can be written
\begin{equation*}
D^pf(u) = \sum_{j=1}^p \sum_{\substack{\alpha_1 + \cdot \cdot \cdot + \alpha_j = p \\ \alpha_i \geq 1}} C_{\alpha_1 , ... , \alpha_j} f^{(j)}(u) (D^{\alpha_1}u) \cdot \cdot \cdot (D^{\alpha_j}u).
\end{equation*}
One can bound the $L^2$ norms (or the Slobodeckij norm when $m$ is not an integer) of the above expression by H\"{o}lder's inequality combined with the Sobolev embeddings of the factors into either continuous functions or appropriate $L^p$ functions. This shows the map is well-defined into $H^m(\T)$. Next, the Fr\'{e}chet derivative is given by multiplication by $f'(u(x))$, a linear map on $H^m(\T)$ that is well defined and continuous when $H^m(\T)$ is an algebra, that is, $m > 1/2$.
\end{proof}


\subsubsection{Differential Operator}\ \\

Our aim is to show $a \to \Xi(a) : J_{1/2, \gamma}^{m,\sigma}(\Pi) \to \widetilde{J}_{0, \gamma}^{m-2, \sigma}(\Pi)$ is analytic near $a = \psi^{1/2}$, where operator $\Xi$ is given by the expression:
\begin{equation}
\Xi(a) = -\frac{1}{a_\psi^3} \Big( 1 + \frac{a_\theta^2}{a^2} \Big) a_{\psi \psi} + 2 \Big( \frac{a_\theta}{a^2 a_\psi^2} \Big) a_{\psi \theta} - \Big( \frac{1}{a^2 a_\psi}\Big) a_{\theta \theta} + \frac{1}{a a_\psi}.
\end{equation}
Notice, this map is a rational function of derivatives of $a(\psi,\theta)$, in other words a superposition map $a \to f(a, a_\psi, a_\theta, a_{\psi \psi}, a_{\psi \theta}, a_{\theta \theta})$, defined by a rational function $f$. The trouble is that these derivatives have distinct leading term asymptotics $\psi^\lambda$, defined by different weights of $J_{\lambda,\gamma}^{m, \sigma}(\Pi)$. It is hopeless to expect any general results of superposition operators on such spaces, regardless of $m$. One expects an exception to this observation when $\lambda = 0$, and thus the leading term is order $1$. We are thus motivated to rewrite $\Xi$ as an operator on $J_{0, \gamma}^{m-2, \sigma}(\Pi)$. 

To do this, we exploit the fact that $a \to \psi^\alpha a : J_{1/2,\gamma}^{m,\sigma}(\Pi) \to J_{1/2+\alpha, \gamma}^{m,\sigma}(\Pi)$ is an isomorphism. In particular we observe, if $a \in J_{1/2, \gamma}^{m, \sigma}(\Pi)$, then each of the following functions belongs to $J_{0, \gamma}^{m-2 ,\sigma}(\Pi)$: 
\begin{equation*}
[ \psi^{-1/2}a ] \text{ , } [ \psi^{1/2} a_\psi ]  \text{ , }  [ \psi^{-1/2}a_\theta ] \text{ , } [ \psi^{-1/2}a_{\theta \theta} ] \text{ , } [\psi^{1/2} a_{\psi \theta}] \text{ , } [\psi^{3/2} a_{\psi \psi}].
\end{equation*}
Writing, $a = \psi^{1/2} [ \psi^{-1/2}a ]$, $a_\psi = \psi^{-1/2} [ \psi^{1/2} a_\psi ]$, etc, and substituting into $\Xi(a)$, we find
\begin{align*}
\Xi(a) & =  -\frac{1}{[ \psi^{1/2} a_\psi ]^3} \Big( 1 + \frac{[ \psi^{-1/2}a_\theta ]^2}{[ \psi^{-1/2}a ]^2} \Big) [\psi^{3/2} a_{\psi \psi}] + 2\frac{[ \psi^{-1/2}a_\theta ] [\psi^{1/2} a_{\psi \theta}]}{[ \psi^{-1/2}a ]^2  [ \psi^{1/2} a_\psi ]^2} \\
& - \frac{ [ \psi^{-1/2}a_{\theta \theta} ]}{[ \psi^{-1/2}a ]^2 [ \psi^{1/2} a_\psi ]} + \frac{1}{[ \psi^{-1/2}a ]  [ \psi^{1/2} a_\psi ]}.
\end{align*}
Notice that all of the $\psi^{\alpha}$ terms outside of the square brackets have cancelled. What remains is a rational function of only square brackets. Each of the square brackets is a multiplication and derivative of $a(\psi,\theta)$ lying in $J_{0, \gamma}^{m-2, \sigma}(\Pi)$, thus analytically depends on $a(\psi,\theta) \in J_{1/2, \gamma}^{m ,\sigma}(\Pi)$. The rational function defining $\Xi$ as written above is analytic so long as the denominator is never zero, that is so long as $\psi^{-1/2}a \neq 0$ and $\psi^{1/2} a_\psi \neq 0$. Since $\Xi(a)$ is now a sum of reciprocals and products of functions in $J_{0,\gamma}^{m-2, \sigma}(\Pi)$, it is enough to prove that the maps 
\begin{gather*}
(u,v) \to uv : J_{0,\gamma}^{m, \sigma}(\Pi) \times J_{0,\gamma}^{m, \sigma}(\Pi) \to J_{0,\gamma}^{m, \sigma}(\Pi), \\
u \to \frac{1}{u} : J_{0,\gamma}^{m, \sigma}(\Pi) \to J_{0,\gamma}^{m, \sigma}(\Pi)
\end{gather*}
are well defined and analytic. We must first prove that the space is an algebra. We start with some useful lemmas in the Kondratev spaces.

\begin{lemma}\label{Embedding Kondratev}
$K_{\gamma}^1(\Pi) \subset L^p(\Pi)$ when $\gamma > \frac{1}{2} - \frac{1}{p}$, with $ \lVert u \rVert_{L^p(\Pi)} \leq C \lVert u \rVert_{K_\gamma^1(\Pi)}$.
\end{lemma}

\begin{proof}
From \ref{Morrey Kondratev}, we have $\lVert u(\psi,\cdot) \rVert_{H^{1/2}(\T)} \leq C \psi^{\gamma-1/2} \lVert u \rVert_{K_\gamma^1(\Pi)}$. By the Sobolev embedding theorem in the critical case, for any $p < \infty$, we have $\lVert u(\psi,\cdot) \rVert_{L^p(\T)} \leq C\lVert u(\psi,\cdot) \rVert_{H^{1/2}(\T)}$. This gives 
\begin{equation*}
\lVert u(\psi,\cdot) \rVert_{L^p(\T)}^p = \int_\T \lvert u(\psi,\cdot) \rvert^p \dd{\theta} \leq C \psi^{p(\gamma-1/2)} \lVert u \rVert_{K_\gamma^1(\Pi)}^p.
\end{equation*}
Integrating over $\psi$ gives
\begin{equation*}
\lVert u \rVert_{L^p(\Pi)}^p = \int_0^1 \int_\T \lvert u(\psi,\cdot) \rvert^p \dd{\theta} \dd{\psi} \leq C  \lVert u \rVert_{K_\gamma^1(\Pi)}^p \int_0^1 \psi^{p(\gamma-1/2)} \dd{\psi}.
\end{equation*}
The right side is bounded when $p(\gamma - \frac{1}{2}) + 1 > 0 $, or $\gamma > \frac{1}{2} - \frac{1}{p}$.
\end{proof}

\begin{lemma}\label{bound on I}
Suppose $u \in K_\gamma^1(\Pi)$. Then $I(\theta) = \big( \int_0^1 \lvert \psi^{-\gamma} u(\psi,\theta) \rvert^2 \dd{\psi} \big)^{1/2 }\in C(\T)$ and $\lvert I \rvert \leq C \lVert u \rVert_{K_\gamma^1(\Pi)}$.
\end{lemma}

\begin{proof}
For $u(\psi,\theta) \in K_\gamma^1(\Pi)$, define the vector-valued map $\theta \to \psi^{-\gamma}u(\psi,\theta)$. This map belongs to $H^1(\T, L^2[0,1])$ because
\begin{align*}
 \lVert \psi^{-\gamma} u(\psi,\theta) & \rVert_{H^1(\T, L^2[0,1])}^2 \\ & = \Big\lVert \lVert \psi^{-\gamma} u(\psi,\theta) \rVert_{L^2[0,1]} \Big\rVert_{L^2(\T)}^2 + \Big\lVert \lVert \psi^{-\gamma} u_\theta(\psi,\theta) \rVert_{L^2[0,1]} \Big\rVert_{L^2(\T)}^2 \\
& \leq \lVert \psi^{-\gamma}u(\psi,\theta) \rVert_{L^2(\Pi)}^2 + \lVert \psi^{-\gamma}u_\theta(\psi,\theta) \rVert_{L^2(\Pi)}^2 \\
& \leq \lVert u \rVert_{K_\gamma^1(\Pi)}^2,
\end{align*}
which follows from the equivalence $L^2(\T, L^2[0,1] ) \cong L^2(\Pi)$. 

By the embedding $H^1(\T, L^2[0,1]) \subset C(\T, L^2[0,1])$, the function  $I(\theta)$ is a composition of continuous maps $\theta \to \psi^{-\gamma}u(\psi,\theta) : \T \to L^2[0,1]$ and $ \lVert \cdot \rVert_{L^2[0,1]} : L^2[0,1] \to \R$ and is thus continuous. Finally,
\begin{align*}
\lvert I(\theta) \rvert & = \lVert \psi^{-\gamma}u(\cdot,\theta) \rVert_{L^2[0,1]} \leq \lVert \psi^{-\gamma}u \rVert_{C(\T, L^2[0,1])} \leq C \lVert \psi^{-\gamma}u \rVert_{H^1(\T, L^2[0,1])} \\ & \leq C \lVert u \rVert_{K_\gamma^1(\Pi)}.
\end{align*}
\end{proof}
 
\begin{lemma}\label{KondratevAlgebra}\ \\
Let $m > 1$ and $\gamma > 1/2$. Then $K_{\gamma}^m(\Pi)$ is an algebra with $\lVert u v \rVert_{K_\gamma^m(\Pi)} \leq C \lVert u \rVert_{K_\gamma^m(\Pi)} \lVert v \rVert_{K_\gamma^m(\Pi)}.$
\end{lemma}
 
\begin{proof}
Suppose $m > 1$ and $\gamma > 1/2$. Then $K_\gamma^m(\Pi) \subset C(\overline{\Pi})$ and these functions vanish at $\psi = 0$. Take $u, v \in K_\gamma^m(\Pi)$. We must show $uv \in K_\gamma^m(\Pi)$. By the product rule, we can write
\begin{equation*}
\partial_\psi^p \partial_\theta^q ( uv ) = \sum_{p'=0}^p \sum_{q'=0}^q C_{p,q}^{p',q'} \partial_\psi^{p-p'} \partial_\theta^{q-q'} ( u ) \partial_\psi^{p'} \partial_\theta^{q'} ( v ).
\end{equation*}
We thus get
\begin{align*}
\lVert uv \rVert_{K_\gamma^m(\Pi)}^2 & = \sum_{p+q=0}^m \lVert \psi^{p-\gamma} \partial_\psi^p \partial_\theta^q (uv) \rVert_{L^2(\Pi)}^2 \\
& \leq C \sum_{p+q=0}^m \sum_{p'=0}^p \sum_{q'=0}^q \lVert \psi^{p-\gamma} \partial_\psi^{p-p'} \partial_\theta^{q-q'} ( u ) \partial_\psi^{p'} \partial_\theta^{q'} ( v ) \rVert_{L^2(\Pi)}^2.
\end{align*}
To bound each of the terms above, we make use of the following estimate (from \ref{Morrey Kondratev})
\begin{equation*}
\lvert \partial_\psi^p \partial_\theta^q u \rvert \leq C \psi^{\gamma-1/2-p} \lVert u \rVert_{K_\gamma^m(\Pi)} \leq C \psi^{-p} \lVert u \rVert_{K_\gamma^m(\Pi)},
\end{equation*}
which holds for $m - (p+q) > 1$, $\gamma > 1/2$. 

First we consider the case when $m-(p-p'+q-q') > 1$.  Then the previous estimate yields $\lvert \partial_\psi^{p-p'} \partial_\theta^{q-q'} u \rvert \leq C \psi^{p'-p} \lVert u \rVert_{K_\gamma^m(\Pi)}$ which then gives
\begin{align*}
\lVert \psi^{p-\gamma} \partial_\psi^{p-p'} \partial_\theta^{q-q'} ( u ) \partial_\psi^{p'} \partial_\theta^{q'} ( v ) \rVert_{L^2(\Pi)} & \leq C \lVert u \rVert_{K_\gamma^m(\Pi)} \lVert \psi^{p'-\gamma} \partial_\psi^{p'} \partial_\theta^{q'} ( v ) \rVert_{L^2(\Pi)}\\ & \leq C \lVert u \rVert_{K_\gamma^m(\Pi)} \lVert v \rVert_{K_\gamma^m(\Pi)}.
\end{align*}
The analogous argument holds if $m - (p'+q') > 1$. It thus remains to consider the case when both $m - (p-p'+q-q') \leq 1$ and $ m -(p'+q') \leq 1$. If we sum these two inequalities we get $2m - (p+q) \leq 2$. Rearranging gives $2m-2 \leq p+q$.  But $p+q \leq m$. We conclude that $2m-2 \leq m$ and so $m \leq 2$. Since we must have $m>1$, this leaves only $m=2$. Thus for $m > 2$, the statement is proven. In the case when $m=2$, the only terms where the above arguments don't apply are
\begin{equation*}
\lVert \psi^{2-\gamma} (\partial_\psi u ) (\partial_\psi v) \rVert_{L^2(\Pi)}, \quad \lVert \psi^{1-\gamma} (\partial_\psi u ) (\partial_\theta v) \rVert_{L^2(\Pi)}, \quad \lVert \psi^{-\gamma} (\partial_\theta u ) (\partial_\theta v) \rVert_{L^2(\Pi)}.
\end{equation*}
We now apply H\"{o}lder's inequality to get the following three inequalities:
\begin{align*}
& \lVert \psi^{2-\gamma} (\partial_\psi u ) (\partial_\psi v) \rVert_{L^2}  = \lVert ( \psi^{1-\frac{\gamma}{2}} \partial_\psi u ) ( \psi^{1-\frac{\gamma}{2}} \partial_\psi v ) \rVert_{L^2}   \leq \lVert \psi^{1-\frac{\gamma}{2}} \partial_\psi u \rVert_{L^4}\lVert \psi^{1-\frac{\gamma}{2}} \partial_\psi v \rVert_{L^4}, \\
& \lVert \psi^{1-\gamma} (\partial_\psi u ) (\partial_\theta v) \rVert_{L^2(\Pi)}  = \lVert ( \psi^{1-\frac{\gamma}{2}} \partial_\psi u ) ( \psi^{-\frac{\gamma}{2}} \partial_\theta v ) \rVert_{L^2}   \leq \lVert \psi^{1-\frac{\gamma}{2}} \partial_\psi u \rVert_{L^4}\lVert \psi^{-\frac{\gamma}{2}} \partial_\theta v \rVert_{L^4}, \\
& \lVert \psi^{-\gamma} (\partial_\theta u ) (\partial_\theta v) \rVert_{L^2} = \lVert ( \psi^{-\frac{\gamma}{2}} \partial_\theta u ) ( \psi^{-\frac{\gamma}{2}} \partial_\theta v ) \rVert_{L^2} \leq \lVert \psi^{-\frac{\gamma}{2}} \partial_\theta u \rVert_{L^4}\lVert \psi^{-\frac{\gamma}{2}} \partial_\theta v \rVert_{L^4}.
\end{align*}
Notice, function $\psi^{1-\gamma/2}\partial_\psi u$ and $\psi^{-\gamma/2} \partial_\theta u$ belong to $K_{\gamma/2}^1(\Pi)$. By \ref{Embedding Kondratev}, these functions belong to $L^4(\Pi)$ when $\frac{\gamma}{2} > \frac{1}{2} - \frac{1}{4}$, which is precisely when $\gamma > 1/2$. We thus get the estimates:
\begin{gather*}
\lVert \psi^{2-\gamma} (\partial_\psi u ) (\partial_\psi v) \rVert_{L^2(\Pi)} \leq C \lVert u \rVert_{K_\gamma^2(\Pi)}  \lVert v \rVert_{K_\gamma^2(\Pi)}, \\
\lVert \psi^{1-\gamma} (\partial_\psi u ) (\partial_\theta v) \rVert_{L^2(\Pi)} \leq C \lVert u \rVert_{K_\gamma^2(\Pi)}  \lVert v \rVert_{K_\gamma^2(\Pi)}, \\
\lVert \psi^{-\gamma} (\partial_\theta u ) (\partial_\theta v) \rVert_{L^2(\Pi)} \leq C \lVert u \rVert_{K_\gamma^2(\Pi)}  \lVert v \rVert_{K_\gamma^2(\Pi)}.
\end{gather*}
This proves the $m=2$ case and thus concludes the proof of the proposition and establishes the bound $\lVert u v \rVert_{K_\gamma^m(\Pi)} \leq C \lVert u \rVert_{K_\gamma^m(\Pi)} \lVert v \rVert_{K_\gamma^m(\Pi)}$.
\end{proof}

\begin{lemma}\ \\
Let $m > 1/2$. Given $\xi(\theta) \in H^m(\T)$ and $u(\psi,\theta) \in K_\gamma^m(\Pi)$, then $\xi u \in K_\gamma^m(\Pi)$ with $\lVert \xi u \rVert_{K_\gamma^m(\Pi)} \leq C \lVert \xi \rVert_{H^m(\T)} \lVert u \rVert_{K_\gamma^m(\Pi)}$.
\end{lemma}

\begin{proof}
By product rule, we have
\begin{align*}
\lVert \xi u \rVert_{K_\gamma^m(\Pi)}^2 &  = \sum_{p+q=0}^m \lVert \psi^{p-\gamma} \partial_\psi^p \partial_\theta^q ( \xi u) \rVert_{L^2(\Pi)}^2 \\ &  \leq C \sum_{p+q=0}^m \sum_{q'=0}^q \lVert \psi^{p-\gamma} \big(D^{q-q'} \xi \big) \big(\partial_\psi^p \partial_\theta^{q'} u \big) \rVert_{L^2(\Pi)}^2.
\end{align*}
For $m> 1/2$, because $\xi(\theta) \in H^m(\T)$, we have $\lvert D^{q-q'} \xi \rvert \leq C \lVert \xi \rVert_{H^m(\T)}$ when $q-q' < m$. In this case, we immediately get
\begin{align*}
\lVert \psi^{p-\gamma} \big(D^{q-q'} \xi \big) \big(\partial_\psi^p \partial_\theta^{q'} u \big) \rVert_{L^2(\Pi)} & \leq C \lVert D^{q-q'}\xi \rVert_{\infty} \lVert \psi^{p-\gamma} \partial_\psi^p \partial_\theta^{q'} u \rVert_{L^2(\Pi)} \\ & \leq C \lVert \xi \rVert_{H^m(\T)} \lVert u \rVert_{K_\gamma^m(\Pi)}.
\end{align*}
In the case when $q-q' = m$, that is $p=0$, $q=m$, $q'=0$, by \ref{bound on I} (since $m \geq 1$), we get
\begin{align*}
\lVert \psi^{-\gamma} (D^m \xi) u \rVert_{L^2(\Pi)} & =  \Big( \int_\T \lvert D^m \xi \rvert^2 \big( \int_0^1 \lvert \psi^{-\gamma} u \rvert^2 \dd{\psi} \big) \dd{\theta} \Big)^{1/2}\\
& \leq \lVert I \rVert_\infty \lVert D^m \xi \rVert_{L^2(\T)}  \\ 
& \leq C \lVert \xi \rVert_{H^m(\T)} \lVert u \rVert_{K_\gamma^m(\Pi)}.
\end{align*}
Thus we have shown $\lVert \xi u \rVert_{K_\gamma^m(\Pi)} \leq C \lVert \xi \rVert_{H^m(\T)} \lVert u \rVert_{K_\gamma^m(\Pi)}$.
\end{proof}

\begin{proposition}\ \\
Ley $m > 1$ and $\gamma > 1/2$. Then $J_{0, \gamma}^{m, \sigma}(\Pi)$ is an algebra with $\lVert u v \rVert_{J_{0, \gamma}^{m, \sigma}(\Pi)} \leq C \lVert u \rVert_{J_{0, \gamma}^m(\Pi)} \lVert v \rVert_{J_{0, \gamma}^{m, \sigma}(\Pi)}.$
\end{proposition}
\begin{proof}
First, let $ u = v(\theta) + w(\psi,\theta)$, $\zeta = \xi(\theta) + \eta(\psi,\theta) \in J_{0,\gamma}^{m}(\Pi)$. This means $v, \xi \in H^m(\T)$ and $w, \eta \in K_\gamma^m(\Pi)$. Multiplying, we get $u\zeta = v\xi + v\eta + \xi w + w\eta$. The leading term $v\xi$ is in $H^m(\T)$ because this space is an algebra. The remaining terms belong to $K_\gamma^m(\Pi)$, by the preceding two propositions. Furthermore, from the bounds established previously and the definition of norm of $J_{0,\gamma}^m(\Pi)$, we have the bound
\begin{align*}
\lVert u \zeta \rVert_{J_{0,\gamma}^m(\Pi)}  & = \lVert v\xi \rVert_{H^m(\T)}^2 + \lVert v\eta + \xi w + w \eta \rVert_{K_{\gamma}^m(\Pi)} \\
& \leq  C \big( \lVert v \rVert_{H^m(\T)} \lVert \xi \rVert_{H^m(\T)} + \lVert v \rVert_{H^m(\T)} \lVert \eta \rVert_{K_\gamma^m(\Pi)} \\  &  \qquad+ \lVert \xi \rVert_{H^m(\T)}  \lVert w \rVert_{K_\gamma^m(\Pi)} +  \lVert w \rVert_{K_\gamma^m(\Pi)}  \lVert \eta \rVert_{K_\gamma^m(\Pi)} \big)\\ 
& \leq C \lVert u \rVert_{J_{0,\gamma}^m(\Pi)}\lVert \zeta \rVert_{J_{0,\gamma}^m(\Pi)}.
\end{align*}
This confirms $J_{0, \gamma}^m(\Pi)$ is an algebra. In the case when $u, \zeta \in J_{0,\gamma}^{m,\sigma}(\Pi)$, we have $v,\xi \in X_\sigma^m(\T)$ and $w,\eta \in K_\gamma^{m, \sigma}(\Pi)$. Then $v \xi \in X_\sigma^m(\T)$ because it is the product of two holomorphic functions in $\T_\sigma$ and so holomorphic itself, and since $H^m(\T)$ is an algebra, we get
\begin{align*}
& \lVert v \xi \rVert_{X_\sigma^m(\T)}  = \lVert v(\cdot+i\sigma) \xi(\cdot + i\sigma) \rVert_{H^m(\T)} + \lVert v(\cdot-i\sigma) \xi(\cdot - i\sigma) \rVert_{H^m(\T)} \\ 
& \leq C \lVert v(\cdot+i\sigma) \rVert_{H^m(\T)}\lVert \xi(\cdot + i\sigma) \rVert_{H^m(\T)} + C\lVert v(\cdot-i\sigma) \rVert_{H^m(\T)}\lVert \xi(\cdot - i\sigma) \rVert_{H^m(\T)} \\
& \leq C  \lVert v  \rVert_{X_\sigma^m(\T)}  \lVert \xi \rVert_{X_\sigma^m(\T)}.
\end{align*}
Next, if $\xi(\theta) \in X_\sigma^m(\T)$ and $w \in K_\gamma^{m,\sigma}(\Pi)$, then the map $\theta \to \xi(\theta) w(\cdot,\theta)$ is holomorphic as a map from $\T_\sigma$ to $K_\gamma^m(0,1]$, with bound $ \lVert \xi w \rVert_{K_\gamma^{m,\sigma}(\Pi)} \leq C \lVert \xi \rVert_{X_\sigma^m(\T)} \lVert w \rVert_{K_\gamma^{m,\sigma}(\Pi)}$, analogously obtained as above. Finally, the same argument holds for the product of $w,\eta \in K_\gamma^{m,\sigma}(\Pi)$. It is the product of holomorphic functions $\T_\sigma \to K_\gamma^m(0,1]$, the latter of which is an algebra and so this product is also holomorphic as a Banach valued map. The similar bound applies, namely that $ \lVert w \eta \rVert_{K_\gamma^{m,\sigma}(\Pi)} \leq C \lVert w \rVert_{K_\gamma^{m,\sigma}(\Pi)} \lVert \eta \rVert_{K_\gamma^{m,\sigma}(\Pi)} $. Putting it all together, we conclude that $J_{0,\gamma}^m(\Pi)$ is an algebra with $ \lVert u \zeta \rVert_{J_{0, \gamma}^{m,\sigma}(\Pi)} \leq C \lVert u \rVert_{J_{0,\gamma}^{m,\sigma}(\Pi)} \lVert \zeta \rVert_{J_{0, \gamma}^{m,\sigma}(\Pi)}$.
\end{proof}

From the previous result along with bilinearity of the multiplication map, one immediately obtains the following result.

\begin{corollary}\label{multiplication is analytic}\ \\
The map $ (u,v) \to uv : J_{0,\gamma}^{m,\sigma}(\Pi) \times J_{0,\gamma}^{m,\sigma}(\Pi) \to J_{0,\gamma}^{m,\sigma}(\Pi)$ is analytic for $m>1$, $\gamma> 1/2$.
\end{corollary}

Having determined that multiplication in $J_{0,\gamma}^{m,\sigma}(\Pi)$ defines an analytic operator, we turn to the operator $u \to \frac{1}{u}$ on $J_{0,\gamma}^{m, \sigma}(\Pi)$. In fact, we consider the more general superposition operator $u \to f(u)$.

\begin{theorem} \label{main superposition}\ \\
Let $f \in C^{m+1}$ on a domain containing image of $u$. Then $u \to f(u) : J_{0,\gamma}^m(\Pi) \to J_{0, \gamma}^m(\Pi)$ is well defined and continuous for $m>1$, $\gamma > 1/2$. Furthermore, if $f \in C^{m+2}(\Omega)$ then this map is $C^1$.
\end{theorem}

\begin{proof}
Given $u(\psi,\theta) \in J_{0,\gamma}^m(\Pi)$, we write $u(\psi,\theta) = \xi(\theta) + v(\psi,\theta)$, where $\xi \in H^m(\T)$ and $v \in K_\gamma^m(\Pi)$. Recall that for $m>1$, $\gamma > 1/2$, $u$ is continuous and $v \to 0$ as $\psi \to 0^+$. In other words, $\xi$ defines the behaviour of $u$ along $\psi=0$, so we write $u(0,\theta) = \xi(\theta)$. By continuity of $f$, we have $f(\xi+v) \to f(\xi)$ as $\psi \to 0^+$. So the behaviour of $f(u)$ at $\psi=0$ is defined by $f(\xi)$. We thus have the decomposition  $f(\xi + v) = f(\xi) + f(\xi + v) - f(\xi)$. Since $\xi \in H^m(\T)$, then by \ref{superposition sobolev}, $f(\xi) \in H^m(\T)$. This forms the leading term of $f(u) \in J_{0,\gamma}^m(\Pi)$. The main task then is to prove that the remainder term,  $f(\xi + v) - f(\xi)$, belongs to $K_\gamma^m(\Pi)$. Intuitively, this means that this term vanishes as $\psi \to 0^+$ at the same rate as $v$ does. We must bound
\begin{equation*}
\lVert f(\xi + v) - f(\xi) \rVert_{K_\gamma^m(\Pi)}^2 = \sum_{p+q=0}^m \lVert \psi^{p-\gamma} \partial_\psi^p \partial_\theta^q \big( f(\xi + v) - f(\xi) \big) \rVert_{L^2(\Pi)}^2.
\end{equation*}
To start, given a composition $f\big(u(\psi,\theta)\big)$, we have the following expressions for its partial derivatives:
\begin{equation*}
\partial_\psi^p \partial_\theta^q f(u) = \sum_{j=1}^{p+q} \sum_{\substack{\alpha_1 + \cdot \cdot \cdot + \alpha_j = p \\ \beta_1 + \cdot \cdot \cdot + \beta_j = q \\ \alpha_i + \beta_i \geq 1 }} C_{\alpha_1, ... , \alpha_j}^{\beta_1, ... , \beta_j} \big( \partial_\psi^{\alpha_1} \partial_\theta^{\beta_1} u \big) \cdot \cdot \cdot \big( \partial_\psi^{\alpha_j} \partial_\theta^{\beta_j} u \big) f^{(j)}(u).
\end{equation*}
If $p\geq 1$, then $\partial_\psi^p \partial_\theta^q f(\xi) = 0$. In this case, with $u = \xi(\theta) + v(\psi,\theta)$, we obtain:
\begin{equation*}
\partial_\psi^p \partial_\theta^q f(\xi+v) = \sum_{j=1}^{p+q} \sum_{\substack{\alpha_1 + \cdot \cdot \cdot + \alpha_j = p \\ \beta_1 + \cdot \cdot \cdot + \beta_j = q \\ \alpha_i + \beta_i \geq 1 }} \sum_{\text{or} }C_{\alpha_1, ... , \alpha_j}^{\beta_1, ... , \beta_j} \big( \partial_\psi^{\alpha_1} \partial_\theta^{\beta_1} \xi \text{ or } v \big) \cdot \cdot \cdot \big( \partial_\psi^{\alpha_j} \partial_\theta^{\beta_j} \xi \text { or } v \big) f^{(j)}(\xi+v),
\end{equation*}
where the summation over `or' indicates we sum over all choices of $\xi$ or $v$ in the above factors. Note though, since $p \geq 1 $, at least some $\alpha_i \neq 0$ and thus the case when all factors choose $\xi$ vanishes. If on the other hand $p=0$, then we instead obtain the expression
\begin{align*}
& \partial_\theta^q \Big(f(\xi+v) - f(\xi) \Big) \\ &= \sum_{j=1}^q \sum_{\substack{\alpha_1+ \cdot \cdot \cdot + \alpha_j =q \\ \alpha_i \geq 1}} C_{\alpha_1, ..., \alpha_j} \Bigg[  \Big( f^{(j)}(\xi+v) - f^{(j)}(\xi) \Big) \left( D^{\alpha_1} \xi \right) \cdot \cdot \cdot \left( D^{\alpha_j} \xi \right) \\
& \quad + \sum_{\text{or}} f^{(j)}(\xi+v) \left( D^{\alpha_1} \xi \text{ or } v \right) \cdot \cdot \cdot \left( D^{\alpha_j} \xi \text{ or } v \right) \Bigg]
\end{align*}
where summation over `or' excludes case when all factors choose $\xi$.

Let us now start with bounds for the case $p \geq 1$. Namely, for $1 \leq j \leq p+q \leq m$, $\alpha_1 + \cdot \cdot \cdot + \alpha_j = p$, $\beta_1 + \cdot \cdot \cdot + \beta_j = q$, $\alpha_i + \beta_i \geq 1$, we must bound 
\begin{equation*}
A = \lVert \psi^{p-\gamma} \big( \partial_\psi^{\alpha_1} \partial_\theta^{\beta_1} \xi \text{ or } v \big) \cdot \cdot \cdot \big( \partial_\psi^{\alpha_j} \partial_\theta^{\beta_j} \xi \text { or } v \big) f^{(j)}(\xi+v) \rVert_{L^2(\Pi)}.
\end{equation*}
Since $f \in C^{m+1}(\Omega)$, we immediately get
\begin{equation*}
A \leq  \lVert f^{(j)} \rVert_\infty \lVert \psi^{p-\gamma} \big( \partial_\psi^{\alpha_1} \partial_\theta^{\beta_1} \xi \text{ or } v \big) \cdot \cdot \cdot \big( \partial_\psi^{\alpha_j} \partial_\theta^{\beta_j} \xi \text { or } v \big) \rVert_{L^2(\Pi)}.
\end{equation*}
Next, $\xi \in H^m(\T)$ and thus $\partial_\theta^{\beta_i} \xi \in C(\T)$ unless $\beta_i = m$. This occurs only if $q=m$ and thus $p=0$, which is outside of the current case $p \geq 1$. Thus we can assume all $\beta_i \leq m$ and so each factor $\partial_\theta^{\beta_i} \xi$ is continuous, and thus can be factored out of the norm. There are $j$ factors of form $\partial_\psi^{\alpha_j} \partial_\theta^{\beta_j} (\xi \text { or } v)$, but at most $j-1$ of them choose $\xi$. We thus get
\begin{equation*}
A \leq  \lVert f^{(j)} \rVert_\infty \lVert \xi \rVert_{H^m(\T)}^{\lambda}\lVert \psi^{p-\gamma} \underbrace{\big( \partial_\psi^{\alpha_1} \partial_\theta^{\beta_1} v \big) \cdot \cdot \cdot \big( \partial_\psi^{\alpha_j} \partial_\theta^{\beta_j} v \big)}_\text{ $\lambda$ terms missing} \rVert_{L^2(\Pi)},
\end{equation*}
where $0 \leq \lambda \leq j-1$. Next, if $m-(\alpha_i + \beta_i) > 1$, then
\begin{equation*}
\lvert \partial_\psi^{\alpha_i} \partial_\theta^{\beta_i} v \rvert \leq C \psi^{\gamma-1/2-\alpha_i} \lVert v \rVert_{K_\gamma^m(\Pi)} \leq C \psi^{-\alpha_i} \lVert v \rVert_{K_\gamma^m(\Pi)},
\end{equation*}
since $\gamma > 1/2$. This condition is not satisfied only when $\alpha_i + \beta_i = m-1$ or $m$. First suppose $\alpha_i + \beta_i = m$. Then $j=1$ and $\alpha_1 = p$ and immediately we get
\begin{equation*}
A \leq \lVert f^{(j)} \rVert_\infty \lVert \psi^{p-\gamma}  \partial_\psi^{\alpha_1} \partial_\theta^{\beta_1} v \rVert_{L^2(\Pi)} \leq \lVert f^{(j)} \rVert_\infty \lVert v \rVert_{K_\gamma^m(\Pi)} \leq \lVert f^{(j)} \rVert_\infty \lVert u \rVert_{J_{0,\gamma}^m(\Pi)}
\end{equation*}
Next, assume without loss of generality that $\alpha_1 + \beta_1 = m-1$. Then either $j=1$, and the same estimate as above holds, or $j=2$. Either $\lambda = 1$ and immediately 
\begin{align*}
A \leq \lVert f^{(j)} \rVert_\infty \lVert \xi \rVert_{H^m(\T)} \lVert \psi^{p-\gamma}  \partial_\psi^{\alpha_1} \partial_\theta^{\beta_1} v \rVert_{L^2(\Pi)} & \leq \lVert f^{(j)}   \rVert_\infty \lVert \xi \rVert_{H^m(\T)} \lVert v \rVert_{K_\gamma^m(\Pi)} \\ & \leq \lVert f^{(j)} \rVert_\infty \lVert u \rVert_{J_{0,\gamma}^m(\Pi)}^2,
\end{align*}
or $\lambda =0$. Necessarily $(\alpha_2 , \beta_2) = (1,0)$ or $(0,1)$. If $m- (\alpha_2 + \beta_2) =  m-1 > 1$, then
\begin{equation*}
\lvert \partial_\psi^{\alpha_2} \partial_\theta^{\beta_2} v \rvert \leq C \psi^{\gamma-1/2-\alpha_2} \lVert v \rVert_{K_\gamma^m(\Pi)} \leq C \psi^{-\alpha_2} \lVert v \rVert_{K_\gamma^m(\Pi)}
\end{equation*}
and so
\begin{align*}
A & \leq  \lVert f^{(j)} \rVert_\infty \lVert \psi^{\alpha_1+\alpha_2-\gamma} \big( \partial_\psi^{\alpha_1} \partial_\theta^{\beta_1} v \big) \big( \partial_\psi^{\alpha_2} \partial_\theta^{\beta_2} v \big) \rVert_{L^2(\Pi)} \\
& \leq  \lVert f^{(j)} \rVert_\infty \lVert v \rVert_{K_\gamma^m(\Pi)} \lVert \psi^{\alpha_1-\gamma} \big( \partial_\psi^{\alpha_1} \partial_\theta^{\beta_1} v \big) \rVert_{L^2(\Pi)} \\
& \leq  \lVert f^{(j)} \rVert_\infty \lVert u \rVert_{J_{0, \gamma}^m(\Pi)}^2.
\end{align*}
If on the other hand $m-{\alpha_2+\beta_2} = m-1 \leq 1$, then $m \leq 2$, so $ m=2$, and we have $A = \lVert \psi^{2-\gamma} ( \partial_\psi v) (\partial_\psi v) f^{(j)}(u) \rVert_{L^2(\Pi)}$ or $A = \lVert \psi^{1-\gamma} ( \partial_\psi v) (\partial_\theta v) f^{(j)}(u) \rVert_{L^2(\Pi)}$. We have seen in \ref{KondratevAlgebra}, how to bound this expression using H\"{o}lder's inequality and embeddings of $\psi^{-\frac{\gamma}{2}} \partial_\theta v$, $\psi^{1-\frac{\gamma}{2}} \partial_\psi v$ into $L^4(\Pi)$ bu \ref{Embedding Kondratev}. This gives $A \leq  \lVert f^{(j)} \rVert_\infty \lVert u \rVert_{J_{0, \gamma}^m(\Pi)}^2$. 

Finally, in the case when $\alpha_i + \beta_i < m-1$, applying the point-wise estimate $\lvert \partial_\psi^{\alpha_i} \partial_\theta^{\beta_i} v \rvert \leq C \psi^{-\alpha_i} \lVert v \rVert_{K_\gamma^m(\Pi)}$ to all but one factor in the expression
\begin{align*}
A & \leq  \lVert f^{(j)} \rVert_\infty \lVert \xi \rVert_{H^m(\T)}^{\lambda}\lVert \psi^{p-\gamma} \underbrace{\big( \partial_\psi^{\alpha_1} \partial_\theta^{\beta_1} v \big) \cdot \cdot \cdot \big( \partial_\psi^{\alpha_j} \partial_\theta^{\beta_j} v \big)}_\text{ $\lambda$ terms missing} \rVert_{L^2(\Pi)} \\
&\leq  C \lVert f^{(j)} \rVert_\infty \lVert \xi \rVert_{H^m(\T)}^{\lambda} \lVert v \rVert_{K_\gamma^m(\Pi)}^{j-\lambda-1 }\lVert \psi^{\alpha_i-\gamma}  \partial_\psi^{\alpha_i} \partial_\theta^{\beta_i} v \rVert_{L^2(\Pi)} \\
& \leq C \lVert f^{(j)} \rVert_\infty \lVert u \rVert_{J_{0,\gamma}^m(\Pi)}^{j}.
\end{align*}
This concludes the case when $p \geq 1$, where we have found each term of $\lVert \psi^{p-\gamma} \partial_\psi^p \partial_\theta^q \big( f(\xi + v) - f(\xi) \big) \rVert_{L^2(\Pi)}$ is bounded by $C \lVert f^{(j)} \rVert_\infty \lVert u \rVert_{J_{0,\gamma}^m(\Pi)}^j$, for $1 \leq j \leq p+q \leq m$. Thus for $p \geq 1$ we can write
\begin{equation*}
\lVert \psi^{p-\gamma} \partial_\psi^p \partial_\theta^q \big( f(\xi + v) - f(\xi) \big) \rVert_{L^2(\Pi)} \leq C \lVert f \rVert_{C^m} \big( \lVert u \rVert_{J_{0,\gamma}^m(\Pi)}+  \lVert u \rVert_{J_{0,\gamma}^m(\Pi)}^m \big).
\end{equation*}

Now we consider the case when $p=0$. Recall, $\partial_\theta^q \big(f(\xi+v) - f(\xi) \big)$ is a sum of 
\begin{equation*}
\sum_{j=1}^q \sum_{\substack{\alpha_1+ \cdot \cdot \cdot + \alpha_j =q \\ \alpha_i \geq 1}} C_{\alpha_1, ..., \alpha_j}  \big( f^{(j)}(\xi+v) - f^{(j)}(\xi) \big) \left( D^{\alpha_1} \xi \right) \cdot \cdot \cdot \left( D^{\alpha_j} \xi \right)
\end{equation*}
and
\begin{equation*}
\sum_{j=1}^q \sum_{\substack{\alpha_1+ \cdot \cdot \cdot + \alpha_j =q \\ \alpha_i \geq 1}} C_{\alpha_1, ..., \alpha_j} \sum_{\text{or}} f^{j}(\xi+v) \left( D^{\alpha_1} \xi \text{ or } v \right) \cdot \cdot \cdot \left( D^{\alpha_j} \xi \text{ or } v \right).
\end{equation*}
Bounding the latter is identical to the previous case of $p \geq 1$. So we have only the first part to bound. That is, for $1 \leq j \leq q \leq m$, $\alpha_1 + \cdot \cdot \cdot + \alpha_j = q$, $\alpha_i \geq 1$, we must bound 
\begin{equation*}
B = \lVert \psi^{-\gamma} \big( f^{(j)}(\xi+v) - f^{(j)}(\xi) \big) \left( D^{\alpha_1} \xi \right) \cdot \cdot \cdot \left( D^{\alpha_j} \xi \right) \rVert_{L^2(\Pi)}.
\end{equation*}
We use the fundamental theorem of calculus to write
\begin{equation*}
 f^{(j)}(\xi+v) - f^{(j)}(\xi) = \int_0^\psi f^{(j+1)}(u) \partial_t v(t,\theta) \dd{t}.
\end{equation*}
Now we apply Hardy's inequality. Set $-\gamma = \alpha-1$. Then $\alpha= 1 - \gamma < 1/2$ since $\gamma > 1/2$. Also set  $f^{(j+1)}(u) \partial_t v(t,\theta) = t^{-\alpha} g(t,\theta)$. We get
\begin{equation*}
B \leq C \lVert \psi^{1-\gamma} f^{(j+1)}(u) \left( \partial_\psi v \right) \left( D^{\alpha_1} \xi \right) \cdot \cdot \cdot \left( D^{\alpha_j} \xi \right) \rVert_{L^2(\Pi)} 
\end{equation*}
Now suppose each $\alpha_i < m$. Then each factor $D^{\alpha_i} \xi$ is continuous and thus
\begin{equation*}
B \leq C \lVert f^{(j+1)} \rVert_{\infty} \lVert \xi \rVert_{H^m(\T)}^j \lVert \psi^{1-\gamma} \left( \partial_\psi v \right)  \rVert_{L^2(\Pi)} \leq C \lVert f^{(j+1)} \rVert_{\infty} \lVert u \rVert_{J_{0,\gamma}^m(\Pi)}^{j+1}.
\end{equation*}
If on the other hand $\alpha_1 = m$, then necessarily $j=1$. We have $\psi v_\psi \in K_\gamma^{m-1}(\Pi)$ and $m -1 \geq 1$, thus by \ref{bound on I}, $\lVert \psi^{1-\gamma} \partial_\psi v(\psi,\theta) \rVert_{L^2[0,1]}$ is continuous with respect to $\theta$. We thus get
\begin{align*}
B & \leq \lVert f^{(j+1)} \rVert_\infty \lVert \psi^{1-\gamma} \left( \partial_\psi v \right) \left(D^m\xi \right) \rVert_{L^2(\Pi)} \\
& \leq \lVert f^{(j+1)} \rVert_\infty \big( \int_\T \lvert D^m \xi \rvert^2 \int_0^1 \lvert \psi^{1-\gamma} \partial_\psi v \rvert^2 \dd{\psi} \dd{\theta} \big)^{1/2}\\
& \leq C \lVert f^{(j+1)} \rVert_\infty \lVert v \rVert_{K_\gamma^m(\Pi)} \lVert D^m \xi \rVert_{L^2(\T)} \\
& \leq C \lVert f^{(j+1)} \rVert_\infty \lVert u \rVert_{J_{0, \gamma}^m(\Pi)}^2 .
\end{align*}
We have thus established the bounds for $p=0$ case. Together with the prior $p \geq 1$ case, we get
\begin{equation*}
\lVert f(\xi + v) - f(\xi) \rVert_{K_\gamma^m(\Pi)} \leq C \lVert f  \rVert_{C^{m+1}} \left(  \lVert u \rVert_{J_{0, \gamma}^m(\Pi)} +  \lVert u \rVert_{J_{0, \gamma}^m(\Pi)}^{m+1} \right).
\end{equation*}
Combining this with bound
\begin{equation*}
\lVert f(\xi) \rVert_{H^m(\T)} \leq C \lVert f \rVert_{C^m} \left(1 +  \lVert \xi \rVert_{H^m(\T)}^m \right)
\end{equation*}
from \ref{superposition sobolev}, we find our desired bound
\begin{align*}
\lVert f(u) \rVert_{J_{0,\gamma}^m(\Pi)} & = \lVert f(\xi) \rVert_{H^m(\T)} + \lVert f(\xi+v) - f(\xi) \rVert_{K_\gamma^m(\Pi)}\\
& \leq C \lVert f \rVert_{C^{m+1}} \left(  1 +  \lVert u \rVert_{J_{0, \gamma}^m(\Pi)}^{m+1} \right).
\end{align*}

Fr\'{e}chet differentiability follows because $J_{0,\gamma}^m(\Pi)$ is an algebra.
\end{proof}

\begin{corollary}\label{superposition}\ \\
Suppose $f$ is complex analytic on a domain containing image of $u$. Then $u \to f(u) : J_{0, \gamma}^{m, \sigma}(\Pi) \to J_{0,\gamma}^{m,\sigma}(\Pi)$ is complex analytic for $m>1$, $\gamma > 1/2$.
\end{corollary}

\begin{proof}
Let $u\in J_{0,\gamma}^{m, \sigma}(\Pi)$. Write $u = \xi + v$, with $\xi \in X_\sigma^m(\T)$, $v \in K_\gamma^{m,\sigma}(\Pi)$. Now, as in the proof of the prior theorem, write $f(u) = f(\xi) + f(\xi+v) - f(\xi)$. By definition, $\xi(\cdot + it) \in H^m(\T)$ for all $\lvert t \rvert \leq \sigma$. By \ref{superposition sobolev}, $f(\xi(\cdot + it)) \in H^m(\T)$ for $\lvert t \rvert \leq \sigma$. Since $f(\xi)$ is the composition of analytic functions, it is itself analytic in $\T_\sigma$, and so $f(\xi) \in X_\sigma^m(\T)$. 

Next, by definition $u(\cdot, \cdot + it) \in J_{0, \gamma}^m(\Pi)$ for all $\lvert t \rvert \leq \sigma$. From the previous theorem, $g(u) = f(\xi+v) - f(\xi) \in K_{\gamma}^m(\Pi)$ for each fixed $\lvert t \rvert \leq \sigma$. Now fix $z \in \T_\sigma$,  then $u(\cdot, z) \in \C \times K_\gamma^m[0,1]$. All of the prior results in this chapter on $K_\gamma^m(\Pi)$ and $J_{0,\gamma}^m(\Pi)$ likewise apply to $K_\gamma^m[0,1]$ and $J_{0,\gamma}^m[0,1]$. Namely, these spaces are algebras and superposition maps are well defined them. Thus $g(z) = f\big(\xi(z) + v(\cdot, z)\big) - f\big(\xi(z)\big)$ is a  $K_\gamma^m[0,1]$ valued map of $z$. Differentiating gives
\begin{align*}
g'(z)  & = f'\big(\xi(z) + v(\cdot, z)\big)\big(\xi'(z) + \partial_z v(\cdot, z)\big) - f'\big(\xi(z)\big)\xi'(z) \\
& = \Big( f'\big(\xi(z) + v(\cdot,z)\big) -f'\big(\xi(z)\big) \Big) \xi'(z) +f'(\xi) \partial_z v(\cdot, z) \\
& + \Big(f'\big(\xi(z) + v(\cdot,z)\big) - f'(\xi(z) \Big) \partial_z v(\cdot, z).
\end{align*}
The first two terms are products of a $K_\gamma^m[0,1]$ function and scalar, the last term is the product of two $K_\gamma^m[0,1]$, which is itself in $K_\gamma^m[0,1]$, since this space is an algebra. Thus we have showed that $z \to f\big(\xi(z)+v(\cdot,z)\big) - f\big(\xi(z)\big)$ is well defined and complex differentiable, thus analytic. This means $f(\xi+v)-f(\xi) \in K_{\gamma}^{m, \sigma}(\Pi)$. We conclude that $f(u) = f(\xi)+ f(\xi+v)-f(\xi) \in J_{0,\gamma}^{m,\sigma}(\Pi)$.
\end{proof}

\begin{theorem}\ \\
For $m>3$, $\gamma > 1/2$, the map $a \to \Xi(a) : J_{1/2,\gamma}^{m,\sigma}(\Pi) \to \widetilde{J}_{0, \gamma}^{m-2, \sigma}(\Pi)$ is analytic in a neighbourhood of $a = \psi^{1/2}$.
\end{theorem}

\begin{proof}
Earlier in this chapter, we showed we can write
\begin{align*}
\Xi(a) & =  -\frac{1}{[ \psi^{1/2} a_\psi ]^3} \Big( 1 + \frac{[ \psi^{-1/2}a_\theta ]^2}{[ \psi^{-1/2}a ]^2} \Big) [\psi^{3/2} a_{\psi \psi}] + 2\frac{[ \psi^{-1/2}a_\theta ] [\psi^{1/2} a_{\psi \theta}]}{[ \psi^{-1/2}a ]^2  [ \psi^{1/2} a_\psi ]^2} \\
& - \frac{ [ \psi^{-1/2}a_{\theta \theta} ]}{[ \psi^{-1/2}a ]^2 [ \psi^{1/2} a_\psi ]} + \frac{1}{[ \psi^{-1/2}a ]  [ \psi^{1/2} a_\psi ]}.
\end{align*}
Thus $\Xi$ is a composition of maps $a  \to [ \cdot \cdot \cdot ]: J_{1/2,\gamma}^{m,\sigma}(\Pi) \to J_{0,\gamma}^{m-2,\sigma}(\Pi)$, which are linear and thus analytic, and a rational function of the square brackets. Given that the square brackets are valued in  $J_{0,\gamma}^{m-2}(\Pi)$, then by \ref{multiplication is analytic} and \ref{superposition}, this rational function is an analytic map $J_{0,\gamma}^{m-2}(\Pi) \times \cdot \cdot \cdot \times J_{0,\gamma}^{m-2}(\Pi) \to J_{0,\gamma}^{m-2}(\Pi)$ when $m-2 > 1$, so long as the denominator does not vanish. To see it does not, suppose $\lVert a - \psi^{1/2} \rVert_{J_{1/2, \gamma}^{m , \sigma}(\Pi)} < \varepsilon$. By boundedness of multiplication by $\psi^{-1/2}$, we get
\begin{equation*}
\lVert \psi^{-1/2}a -1 \rVert_{J_{0, \gamma}^{m, \sigma}(\Pi)} \leq C \lVert a - \psi^{1/2} \rVert_{J_{1/2, \gamma}^{m , \sigma}(\Pi)} \leq C\varepsilon.
\end{equation*}
Then
\begin{equation*}
\lvert \psi^{-1/2}a - 1 \rvert \leq D \lVert \psi^{-1/2}a-1 \rVert_{J_{0, \gamma}^{m,\sigma}(\Pi)} \leq CD\varepsilon.
\end{equation*}
Taking $\varepsilon$ small enough, we can ensure $\psi^{-1/2}a$ is close enough to $1$ in $\C$ that it is never zero. 
Similarly, by boundedness of $\partial_\psi$,
\begin{equation*}
\lVert a_\psi - 1/2 \psi^{-1/2} \rVert_{J_{-1/2, \gamma}^{m-1, \sigma }(\Pi)} \leq C \lVert a - \psi^{1/2} \rVert_{J_{1/2, \gamma}^{m , \sigma}(\Pi)} \leq C\varepsilon.
\end{equation*}
By boundedness of multiplication by $\psi^{1/2}$,
\begin{equation*}
\lVert \psi^{1/2}a_\psi - 1/2 \rVert_{J_{0, \gamma}^{m-1, \sigma }(\Pi)} \leq D \lVert a_\psi - 1/2 \psi^{-1/2} \rVert_{J_{-1/2, \gamma}^{m-1, \sigma }(\Pi)}.
\end{equation*}
Then
\begin{equation*}
\lvert \psi^{1/2}a_\psi - 1/2 \rvert \leq E \lVert \psi^{1/2}a_\psi - 1/2 \rVert_{J_{0, \gamma}^{m-1, \sigma }(\Pi)} \leq CDE\varepsilon.
\end{equation*}
Again, taking $\varepsilon$ small enough, we can ensure $\psi^{1/2}a_\psi$ is close enough to $1/2$ in $\C$ it is never zero. Thus we have shown $\Xi: U \to J_{0, \gamma}^{m-2,\sigma}(\Pi)$ is an analytic map on a neighbourhood $U \subset J_{1/2, \gamma}^{m,\sigma}(\Pi)$ of $\psi^{1/2}$.

It remains to show that $\Xi$ is in fact $\widetilde{J}_{0, \gamma}^{m-2, \sigma}(\Pi)$ valued. This means the leading $H^m(\T)$ term of $\Xi(a)$ has zero second-order Fourier coefficients. We have seen that for $\gamma > 1/2$, the leading term of $\Xi(a)$ depends only on the leading term of the square brackets, which in turn depend only on the leading term of $a(\psi,\theta)$, that is, depend only on $\psi^{1/2}\xi(\theta)$. Thus we must show $\int_\T \Xi \big( \psi^{1/2}\xi(\theta) \big) e^{\pm 2i\theta} \dd{\theta} =0$. We find
\begin{equation*}
\Xi ( \psi^{1/2} \xi) = \frac{4}{\xi^2} + \frac{6(D\xi)^2}{\xi^4}- \frac{2D^2\xi}{\xi^3}.
\end{equation*}\
We find
\begin{align*}
\int_\T \Xi ( \psi^{1/2} \xi) e^{\pm 2 i \theta} \dd{\theta} & = 2\int_\T \Big( \frac{2}{\xi^2} + \frac{3(D\xi)^2}{\xi^4}- \frac{D^2\xi}{\xi^3}  \Big) e^{\pm 2 i \theta} \dd{\theta} \\
& = 2\int_\T \Big( \frac{2}{\xi^2} \pm 2i \frac{D\xi}{\xi^3}  \Big) e^{\pm 2 i \theta} \dd{\theta} \\
& = 2\int_\T \Big( \frac{2}{\xi^2} \mp iD \big(\frac{1}{\xi^2}\big)   \Big) e^{\pm 2 i \theta} \dd{\theta} \\
& = 0
\end{align*}
where we have integrated by parts the last term on the first line, and again the last term on the third line. We conclude that $\Xi(a) \in \widetilde{J}_{0,\gamma}^{m-2,\sigma}(\Pi)$, thus the statement of the theorem is proved.
\end{proof}


\subsubsection{Boundary Operator}\ \\

We now turn our attention to the nonlinear boundary map
\begin{multline*}
B(b,R,p,a) = -b^2\Big( \arctan \big(p_y + Ra(1,\theta) \sin\theta, p_x + Ra(1,\theta)\cos\theta \big) \Big)\\
 + R^2a^2(1,\theta) + 2Ra(1,\theta) (p_x \cos\theta + p_y\sin\theta) + p_x^2 + p_y^2.
\end{multline*}

\begin{proposition} \label{boundary analytic}\ \\
Suppose $f$ is complex analytic on a domain containing image of $u$. Then $u \to f(u) : X_\sigma^{m-1/2}(\T) \to X_\sigma^{m-1/2}(\T)$ is complex analytic for $m>1$.
\end{proposition}

\begin{proof}
If $u \in X_\sigma^{m-1/2}(\T)$ and $f$ complex analytic on a domain containing image of $u$, then $f(u)$ is a composition of holomorphic functions and thus holomorphic in $\T_\sigma$. Furthermore, $u(\cdot + it) \in H^{m-1/2}(\T)$ for every $\lvert t \rvert \leq \sigma$ and thus by the previous result,
$f \big( u(\cdot + it) \big) \in H^{m-1/2}(\T)$ for every $\lvert t \rvert \leq \sigma$. Thus $f(u) \in X_\sigma^{m-1/2}(\T)$. Additionally, this map is complex differentiable and thus analytic.
\end{proof}

To apply the above result to our boundary map, we define the superposition operators:
\begin{gather}\label{boundary maps}
(R,p,a) \to X(\theta) = p_x + Ra(1,\theta)\cos\theta, \\
(R,p,a) \to Y(\theta) = p_y + Ra(1,\theta)\sin\theta, \\
(R, p,a) \to f(\theta) = R^2a^2(1,\theta) + 2Ra(1,\theta) (p_x \cos\theta + p_y\sin\theta) + p_x^2 + p_y^2, \\
(R,p,a) \to \varphi(\theta) = \arctan \big(p_y + Ra(1,\theta) \sin\theta, p_x + Ra(1,\theta)\cos\theta \big).
\end{gather}
Then the boundary map can be written $(R,p,a) \to B(\theta) = -b^2(\varphi(\theta)) + f(\theta)$. First we establish analyticity of superposition maps $X$, $Y$ and $f$. Second we will address the map $\varphi$ and the composition $b^2(\varphi)$.

\begin{corollary}\ \\
The maps
\begin{equation*}
(R,p,a) \to X,Y,f : \C^3 \times J_{1/2,\gamma}^{m,\sigma}(\Pi) \to X_\sigma^{m-1/2}(\T)
\end{equation*}
defined in \ref{boundary maps} are complex analytic for $m \geq 2$.
\end{corollary}

\begin{proof}
First, the restriction map $a(\psi,\theta) \to a(1,\theta) : J_{1/2,\gamma}^{m, \sigma}(\Pi) \to X_\sigma^{m-1/2}(\T)$ is linear and thus analytic. Multiplication of functions by $\cos\theta$ and $\sin\theta$ is a linear map, well defined into $X_\sigma^{m-1/2}(\T)$ and thus also analytic. In particular, the maps $(R,p,a) \to X,Y, f$ can be viewed as compositions of the linear map
\begin{align*}
(R, p_x,p_y, a(\psi,\theta)) & \to (R, p_x, p_y, p_x \cos\theta, p_y \sin\theta, a(1,\theta)) \\
\C^3 \times J_{1/2,\gamma}^{m, \sigma}(\Pi) & \to \underbrace{X_{\sigma}^{m-1/2}(\T) \times \cdot \cdot \cdot \times X_{\sigma}^{m-1/2}(\T)}_\text{$6$ times},
\end{align*}
and a polynomial on $\C^6$. By \ref{boundary analytic}, these maps are thus analytic  $\C^3 \times J_{1/2,\gamma}^{m, \sigma}(\Pi) \to X_\sigma^{m-1/2}(\T)$, so long as $m \geq 2$.
\end{proof}

Now we must consider the map $(R,p,a) \to \varphi(\theta)$. This maps takes the graph of polar function $r=Ra(1,\theta)$ centered at $p$ and returns the corresponding angle coordinate of this graph in $(\rho,\varphi)$ coordinates centered at the origin. So long as $p$ is close to $0$ and the graph $a(1,\theta)$ is close to a circle so that it corresponds to the graph of a polar function in both coordinates, then this nonlinear coordinate change will be well defined. In the real case, it will be some diffeomorphism of $\T$. In the complex case, we expect a biholomorphism from $\T_\sigma$ to a slightly deformed complex periodic strip $\varphi \{ T_\sigma \}$. Since this deformation should be continuous with respect to $(R,p,a)$, then for any $\tau > \sigma > 0$, we can take $(R,p,a)$ close enough to $(1,0, \psi^{1/2})$ that we get $\varphi \{ \T_\sigma \} \subset \T_\tau$.
 
\begin{proposition}\ \\
Let $m \geq 2$. For any $\tau > \sigma > 0$, there exists $\varepsilon > 0$ small enough such that if $\lvert R - 1 \rvert < \varepsilon$, $\lvert p \rvert < \varepsilon$ and  $\lVert a - \psi^{1/2} \rVert_{J_{1/2,\gamma}^{m,\sigma}(\Pi)} < \varepsilon$, then the map
\begin{equation*}
(R,p,a) \to \varphi(\theta) : \C^3 \times J_{1/2,\gamma}^{m,\sigma}(\Pi) \to X_\sigma^{m-1/2}(\T)
\end{equation*}
is analytic and the image $\varphi \{ \T_\sigma \}$ is contained in $\T_\tau$.
\end{proposition}

\begin{proof}\

In the real case, $(x,y) \in \R^2 \setminus \{ 0 \} \to \varphi = \atan(y,x)$ is a $\T$-valued function giving the angle between the plane vector $(x,y)$ and the $x$-axis. Equivalently, one can think of $\atan(y,x)$ as a (helicoidal) multivalued function with the property that if $\atan(y,x) = \varphi$, then also $\atan(y,x) = \varphi + 2\pi k$ for any $k \in \Z$.

Let us define the usual one-argument arctangent by $\atan(y/x) \in (-\pi/2, \pi/2)$ for $x>0$. Using the one-argument arctangent, we can for example define the following four charts of the multivalued arctangent:
\begin{equation*}
\atan(y,x) = 
	\begin{cases}
	\atan(y/x)  & \text{if } x > 0 \\
	\pi/2-\atan(x/y) & \text{if } y > 0\\
	\pi - \atan(y/x) & \text{if } x < 0\\
	3\pi/2-\atan(x/y) & \text{if } y < 0.
	\end{cases}
\end{equation*}
These charts correspond to values of $\varphi$ in $(-\pi/2, \pi/2)$, $(0, \pi)$, $ (\pi/2, 3\pi/2)$ and $(\pi, 2\pi)$ respectively. Adding $2\pi k$ with $k \in \Z$ defines the remaining charts of the full helicoid.

Next, consider the complexifications $x \to X = x + i \xi$ and $y \to Y = y + i\eta$. Analogous charts of $\atan(Y,X)$ for $x >0$, $y > 0$, $x < 0$ and $y < 0$ are defined by use of the complex one-argument function $\atan(z)$, with $z = X/Y$ or $z = Y/X$. The function $\atan(z)$, which is the complex extension of the real one-argument arctangent with values in $( -\pi/2, \pi/2)$, is analytic except at $\{ z : \Re{z} = 0, \lvert \Im{z} \rvert \geq 1 \}$.

Now, treating $R$, $p$, $a$, $\theta$ as real, define the complex extensions: $p_x \to p_x + i \pi_x$, $p_y \to p_y + i \pi_y $, $Ra(1,\theta) \to \alpha + i\beta $ and $\theta \to \theta + it$. The last extension gives identities 
\begin{equation*}
\cos(\theta + it) = \cos\theta \cosh t - i\sin\theta\sinh t \text{ ,} \quad \sin(\theta + it) = \sin\theta \cosh t + i \cos\theta \sinh t. 
\end{equation*}
These induce complexifications of $p_x + Ra(1,\theta)\cos\theta$ and $p_y + Ra(1,\theta)\sin\theta$, given by
\begin{align*}
X & = (p_x + \alpha \cos\theta \cosh t + \beta \sin\theta \sinh t) + i ( \pi_x - \alpha \sin\theta \sinh t + \beta \cos\theta \cosh t) \\ & = x + i \xi,
\end{align*}
\begin{align*}
Y & = (p_y + \alpha \sin\theta \cosh t - \beta \cos\theta \sinh t) + i ( \pi_y + \alpha \cos\theta \sinh t + \beta \sin\theta \cosh t) \\ & = y + i \eta.
\end{align*}
First suppose $x > 0$. We have
\begin{equation*}
z = \frac{Y}{X} = \frac{(xy + \xi \eta)+ i(x\eta - y\xi)}{x^2+\xi^2}.
\end{equation*}
If $\Re{z} \neq 0$, then $\atan(z)$ is analytic. If on the other hand $\Re{z} = 0$, then $\atan(z)$ is analytic when $\lvert \Im{z} \rvert < 1$. So suppose $\Re{z} = 0$. Then $xy + \xi \eta = 0$. Since $x > 0$, we have $y = -{\xi \eta}/{x}$. From this we find that $\Im{z} = \eta/x$. Thus for $\atan(z)$ to be analytic, we require that $\lvert \eta /x \rvert < 1$, or equivalently, that $ \lvert \eta \rvert < \lvert x \rvert$. 

Substituting expressions for $x$, $y$, $\xi$, $\eta$ into condition $xy + \xi \eta = 0$ and using identity $\cosh^2 t - \sinh^2 t = 1$, we get
\begin{multline*}
p_x p_y + \pi_x \pi_y + (\alpha p_x + \beta \pi_x)\sin\theta \cosh t + (\alpha \pi_x - \beta p_x) \cos \theta \sinh t \\ + (\alpha p_y + \beta \pi_y) \cos \theta \cosh t - (\alpha \pi_y - \beta p_y) \sin \theta \sinh t + (\alpha^2 + \beta^2) \sin\theta \cos\theta = 0.
\end{multline*}
From the statement of the theorem, we have $p_x, p_y, \pi_x, \pi_y, \beta \sim \varepsilon$ and $\alpha \sim 1$. The above equality implies that the last term of the left hand side is of the same order as the other terms, thus we deduce $\sin\theta \cos\theta \sim \varepsilon (\sin\theta + \cos\theta)(\cosh t + \sinh t) \sim \varepsilon$,
since $ \lvert t \rvert < \sigma$ and $\sigma$ is fixed. If $\varepsilon$ is small enough, then $\sin\theta \cos\theta \sim \varepsilon$ implies either $\sin\theta \sim \varepsilon$ or $\cos\theta \sim \varepsilon$. Since we work on the chart $x>0$, we can assume without loss of generality that $\sin\theta \sim \theta \sim \varepsilon$ and thus $\cos \theta \sim 1$. The other case can be handled by charts $y>0$ and $y < 0$.

Returning to the desired estimate $ \lvert \eta \rvert < \lvert x \rvert$, observe $ \eta = \pi_y + \alpha \cos\theta \sinh t + \beta \sin\theta \cosh t$ and $ x = p_x + \alpha \cos \theta \cosh t + \beta \sin \theta \sinh t$. We have
\begin{align*}
\lvert \eta \rvert & = \lvert \pi_y + \alpha \cos\theta \sinh t + \beta \sin\theta \cosh t \rvert\\
& \leq \lvert \pi_y \rvert + \alpha \cos\theta \lvert \sinh t \rvert + \lvert \beta \sin\theta \rvert \cosh t \\
& = \lvert \pi_y \rvert + \lvert p_x \rvert - \lvert p_x \rvert + (\alpha \cos\theta - \lvert \beta \sin\theta \rvert)(\lvert \sinh t \rvert - \cosh t)- \lvert \beta \sin\theta \sinh t \rvert \\
& \phantom{{}={}} + 2\lvert \beta \sin\theta \sinh t \rvert + \alpha \cos\theta \cosh t \\
& < \lvert \pi_y \rvert + \lvert p_x \rvert + 2 \lvert \beta \sin\theta \sinh \sigma \rvert + ( \alpha \cos\theta - \lvert \beta \sin\theta \rvert)( \sinh\sigma - \cosh \sigma) \\
& \phantom{{}={}} + \lvert p_x + \alpha \cos\theta \cosh t + \beta \sin\theta \sinh t \rvert.
\end{align*}
Here we have used the fact that $\lvert t \rvert < \sigma$. Finally, since $p_x$, $\pi_y$, $\beta$, $\sin\theta \sim \varepsilon$ and $\alpha \cos\theta - \lvert \beta \sin\theta \rvert \sim 1$, for any $ \sigma$, we can take $\varepsilon$ small enough such that $\lvert \pi_y \rvert + \lvert p_x \rvert + 2 \lvert \beta \sin\theta \sinh \sigma \rvert + ( \alpha \cos\theta - \lvert \beta \sin\theta \rvert)( \sinh\sigma - \cosh \sigma) < 0$. We thus get $\lvert \eta \rvert < \lvert x \rvert$. 

Analogous arguments hold for the other charts (with $z = X/Y$ for $y > 0$ and $y < 0$). Returning to our standard notation where $R$, $p$ and $a(\psi,\theta)$ are $\C$-valued, we conclude that for any $\sigma > 0$, there exists $\varepsilon > 0$ small enough such that for $\lvert R - 1 \rvert < \varepsilon$, $\lvert p \rvert < \varepsilon$ and  $\lVert a - \psi^{1/2} \rVert_{J_{1/2,\gamma}^{m,\sigma}(\Pi)} < \varepsilon$, the multivalued function $\varphi = \atan(Y,X)$ is analytic on the image of $X(\theta) = p_x + Ra(1,\theta)\cos\theta$, $Y(\theta) = p_y + Ra(1,\theta) \sin\theta$. By \ref{boundary analytic}, for $m\geq2$,
\begin{equation*}
(R,p,a) \to (X, Y) \to \atan(X,Y)
\end{equation*}
\begin{equation*}
 \C^3 \times J_{1/2,\gamma}^{m,\sigma}(\Pi) \to X_\sigma^{m-1/2}(\T) \times X_\sigma^{m-1/2}(\T) \to X_\sigma^{m-1/2}(\T)
\end{equation*}
is the composition of analytic maps and is thus analytic.

Observe that $\varphi = \atan \big(p_y + Ra(1,\theta)\sin\theta, p_x + Ra(1,\theta)\cos\theta \big)$ defines a conformal map of $\theta$ in the periodic strip $\T_\sigma$ which is conformal to an annulus. Thus its image is some deformed periodic strip of equal modulus of annulus. By consequence of the above result, $(R,p,a) \to \varphi : \C^3 \times J_{1/2, \gamma}^{m,\sigma}(\Pi) \to X_\sigma^{m-1/2}(\T)$ is continuous at $(R,p,a) = (1,0, \psi^{1/2})$, where we have $\varphi(1,0, \psi^{1/2}) = \atan(\sin\theta, \cos\theta) = \theta$. By continuity of this map and the embedding $X_\sigma^{m-1/2}(\T) \in C(\overline{\T}_\sigma)$, we can make this deformation arbitrary small. In particular, for any $\tau > \sigma$, we can find $\varepsilon$ small enough such that $ \varphi\{\T_\sigma\} \subset \T_\tau$.

\end{proof}

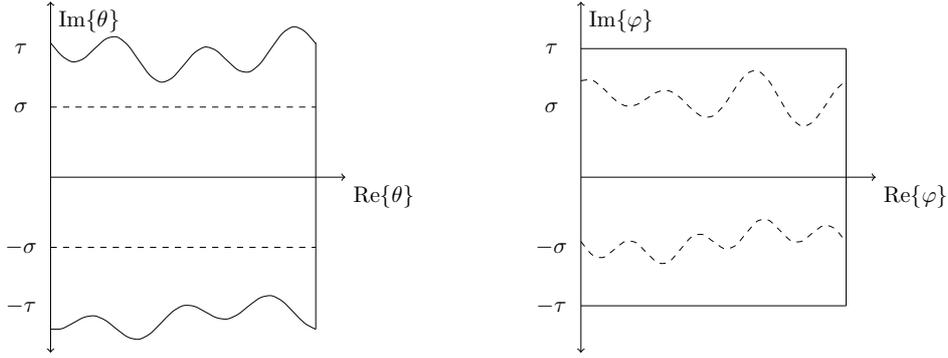
\begin{figure}[H]
\scalebox{0.775}
{
\centering
\begin{tikzpicture}
	\draw[->] (-7,0) -- (-2,0) node[anchor=north west] {$\Re{\theta}$};
	\draw[<->] (-7,-3) -- (-7,3) node[anchor=north west] {$\Im{\theta}$};
	\draw[dashed] (-7,1.2) -- (-2.5,1.2);
	\draw[dashed] (-7,-1.2) -- (-2.5,-1.2);
	\draw (-2.5, -2.6) -- (-2.5, 2.3);%
	\node at (-7.5, 1.2) {$\sigma$};
	\node at (-7.5, -1.2) {$-\sigma$};
	\node at (-7.5, 2.2) {$\tau$};
	\node at (-7.5, -2.2) {$-\tau$};
	\draw[domain=0:4.5, smooth] plot ({-7 + \x}, {2.1+0.2*cos(1.4*(\x r))-0.3*sin(4.2*(\x r))});
	\draw[domain=0:4.5, smooth] plot ({-7+ \x}, {-2.4-0.2*sin(1.4*(\x r))-0.2*cos(4.2*(\x r))});
	\draw[->] (2,0) -- (7,0) node[anchor=north west] {$\Re{\varphi}$};
	\draw[<->] (2,-3) -- (2,3) node[anchor=north west] {$\Im{\varphi}$};
	\draw (2, 2.2) -- (6.5,2.2);
	\draw (2,-2.2) -- (6.5,-2.2);
	\draw (6.5, -2.2) -- (6.5, 2.2);
	\node at (1.5, 2.2) {$\tau$};
	\node at (1.5, -2.2) {$-\tau$};
	\node at (1.5, 1.2) {$\sigma$};
	\node at (1.5, -1.2) {$-\sigma$};
	\draw[domain=0:4.5, smooth, dashed] plot ({2+ \x}, {1.35+0.2*sin(2.8*(\x r))+0.3*cos(4.2*(\x r))});
	\draw[domain=0:4.5, smooth, dashed] plot ({2+ \x}, {-1.1-0.2*sin(1.4*(\x r))-0.2*sin(5.6*(\x r))});
\end{tikzpicture}
}
\caption*{The graph of a polar function can be represented in coordinates $(r,\theta)$ and $(\rho, \varphi)$. The map $\theta \to \varphi$ takes $\T_\sigma$ to some deformed strip (enclosed by the dashed curves on the right side) which is contained in $\T_\tau$. Conversely, $\varphi \to \theta$ takes $\T_\tau$ to some deformed strip (enclosed by the solid curves on the left side) which contains $\T_\sigma$.}
\end{figure}
\begin{remark}
The point is that the nonlinear coordinate change $\theta \leftrightarrow \varphi$ between domains of analyticity is not a self map on $\T_\sigma$. From a reverse perspective, given a prescribed boundary function $\rho = b(\varphi)$ analytic on some domain, the domain of analyticity of $r=Ra(1,\theta)$ will depend on the solution itself (namely the position $p$). Since we require the pool of solutions to be taken from the same Banach space, we must fix the domain of solutions. To work around this, we enlarge the domain of analyticity of prescribed boundary functions $\rho = b(\varphi)$ to $\T_\tau$ with $\tau > \sigma$, so that in a sufficiently small neighbourhood of solution $(R,p,a) = (1,0, \psi^{1/2})$, all solutions map $ \theta \to \varphi : \T_\sigma \to \T_\tau$. That is, we prescribe an analytic boundary function $\rho = b(\varphi)$ whose complex singularities are restricted to $\lvert \Im{\varphi} \rvert \geq \tau$. Then, we describe our solutions on domain $\theta \in \T_\sigma$ with $\rvert \Im{\varphi(\theta)} \lvert < \tau$ so that they do not include the prescribed singularities. For this reason, these prescribed singularities can be of any strength and the boundary functions $\rho=b(\varphi)$ can be taken in any Banach space $\mathrm{H}(\T_\tau)$ of functions holomorphic in $\T_\tau$.
\end{remark}

\begin{theorem}\ \\
Let $m \geq 2$. For any $\tau > \sigma> 0$, there exists a neighbourhood of solution $R=1$, $p=0$ and $a(\psi,\theta) = \psi^{1/2}$ on which 
the boundary map 
\begin{equation*}
(b,R,p,a) \to B : \mathrm{H}(\T_\tau) \times \C^3 \times J_{1/2,\gamma}^{m, \sigma}(\Pi) \to X_\sigma^{m-1/2}(\T)
\end{equation*}
is analytic, for any Banach space $\mathrm{H}(\T_\tau)$ of functions holomorphic in $\T_\tau$.
\end{theorem}

\begin{proof}\ \\
We saw that the composition
\begin{equation*}
(R,p,a) \to (X,Y) \to \varphi
\end{equation*}
\begin{equation*}
\C^3 \times J_{1/2, \gamma}^{m,\sigma}(\Pi) \to X_\sigma^{m-1/2}(\T) \times X_\sigma^{m-1/2}(\T) \to X_\sigma^{m-1/2}(\T)
\end{equation*}
is analytic and the image of $\varphi(\theta)$ is contained in $\T_\tau$. By \ref{boundary analytic}, the map 
\begin{equation*}
(b, \varphi) \to b \circ \varphi : \mathrm{H}(\T_\tau) \times X_\sigma^{m-1/2}(\T) \to X_\sigma^{m-1/2}(\T)
\end{equation*}
is well defined and analytic in $\varphi$. Also it is linear and thus analytic in $b$. Thus it is analytic in the product space $\mathrm{H}(\T_\tau) \times X_\sigma^{m-1/2}(\T)$. Again by \ref{boundary analytic}, the map
\begin{equation*}
b \circ \varphi \to (b \circ \varphi)^2 : X_\sigma^{m-1/2}(\T) \to  X_\sigma^{m-1/2}(\T)
\end{equation*}
is analytic. Finally, we saw also that the map
\begin{equation*}
(R,p,a) \to f : \C^3 \times J_{1/2, \gamma}^{m,\sigma}(\Pi) \to X_\sigma^{m-1/2}(\T)
\end{equation*}
is analytic.Thus, $(b,R,p,a) \to  B = -b^2(\varphi) + f$ is the composition and sum of analytic maps and thus analytic.

\end{proof}


\section{The analytic manifold of stationary flows with an elliptic stagnation point}

The principle driving our work is the representation of a flow as a collection of its flow lines. We have introduced function spaces which describe families of topologically circular flow lines around a single non-degenerate elliptic fixed point. A partial complex analytic structure on these function spaces incorporates the flow line analyticity. In our formulation, stationary flows are governed by a nonlinear degenerate elliptic boundary value problem, which can be expressed as an analytic operator equation in the defined function spaces.

\begin{theorem}[Main Result]\label{Main Result}\ \\
Let $m>3$, $1/2 < \gamma < 1$ and $\tau > \sigma > 0$. There exists a neighbourhood of $F(\psi) =4 $ in $J_{0,\gamma}^{m-2}[0,1]$, $b(\varphi) = 1$ in $\mathrm{H}(\T_\tau)$, $R = 1$ in $\C$, $p=0$ in $\C^2$ and $a(\psi,\theta) = \psi^{1/2}$ in $J_{1/2, \gamma}^{m,\sigma}(\Pi)$, in which \ref{NLBVP} has a unique solution that is parameterized by analytic map 
\begin{equation*}
(F, b) \to (R,p,a) : J_{0,\gamma}^{m-2}[0,1] \times \mathrm{H}(\T_\tau) \to \C^3 \times J_{1/2, \gamma}^{m,\sigma}(\Pi).
\end{equation*}
\end{theorem}

\begin{proof}
Equation \ref{NLBVP} can be written as an operator equation 
\begin{equation*}
( F, b, R, p, a) \to \big( \Xi(a) - F, B \big) = 0
\end{equation*}
between complex Banach spaces 
\begin{equation*}
J_{0,\gamma}^{m-2}[0,1] \times \mathrm{H}(\T_\tau) \times \C^3 \times J_{1/2, \gamma}^{m,\sigma}(\Pi) \to \widetilde{J}_{0,\gamma}^{m-2,\sigma}(\Pi) \times X_\sigma^{m-1/2}(\T).
\end{equation*}
This equation has a solution at $(F,b,R,p,a) = (4,1,1,0,\psi^{1/2})$ and in a neighbourhood of this solution, the above operator is analytic. The linearization
\begin{equation*}
\frac{\partial \big( \Xi(a) - F, B \big)}{\partial (R, p, a)} : \C^3 \times J_{1/2, \gamma}^{m,\sigma}(\Pi) \to \widetilde{J}_{0,\gamma}^{m-2,\sigma}(\Pi) \times X_\sigma^{m-1/2}(\T)
\end{equation*}
at this solution defines a Banach isomorphism. By the analytic implicit function theorem in complex Banach spaces, the result follows.
\end{proof}

Recall that the unknown $R$ was introduced into the solution as an extra degree of freedom to accommodate the fact that specifying $\psi$ at the fixed point (as we have done) yields an overdetermined problem. Under such circumstances, only the vorticity and the `shape' of domain should be treated as parameters, where as the `radius' of domain depends on vorticity. In our construction, the solutions for which $R \neq 1$ are fictitious in that they are produced by incompatible choices of vorticity and domain. Taking the pre-image of solutions with $R=1$ defines a codimension-one submanifold of the parameter space, consisting of precisely the compatible parameters.

\begin{theorem}\ \\
Under the conditions of \ref{Main Result}, in a neighbourhood of the circular flow of constant vorticity, the set of stationary flows having a single, non-degenerate elliptic fixed point form a complex Banach manifold in $J_{1/2, \gamma}^{m,\sigma}(\Pi)$ parameterized by a codimension-one submanifold of $J_{0,\gamma}^{m-2}[0,1] \times \mathrm{H}(\T_\tau)$.
\end{theorem}

Comparing to our previous work \cite{Da} in which we obtained an analytic parameterization of stationary flows in a periodic strip without fixed point, our result here is a touch weaker. The parameterization of the prior work includes in its description the prescribed singularities of the boundary flow lines (which may occur along $\partial \T_\sigma$). In the parameterization provided here, the prescribed singularities are explicitly avoided from the description. The limitation seems only of a technical nature resulting from the coordinate changes induced by translations of the fixed point of the flow. In the case of solutions for which the fixed point does not deviate from the origin, this coordinate change does not occur. We then expect the following strengthening of our main result:

\begin{theorem}\ \\
Suppose $(b,F) \in X_{\tau}^{m-1/2}(\T) \times J_{0, \gamma}^{m-2}[0,1]$ are such that solution $a(\psi,\theta) \in J_{1/2, \gamma}^{m,\sigma}(\Pi)$ has fixed point at $p=0$. Then in fact  $a(\psi,\theta) \in J_{1/2, \gamma}^{m,\tau}(\Pi)$.
\end{theorem}

It remains to be seen how to show this improvement. Doing so would bring our result in exact analogy with the prior work on the periodic strip.

}
\bibliography{Bibliography}
\bibliographystyle{amsplain}

\end{document}